\documentclass[twocolumn]{autart}    

\usepackage{graphicx}          

\usepackage{amsmath}
\usepackage{amsfonts}       
\usepackage{amssymb}
\usepackage{color}
\usepackage{enumerate}
\usepackage{booktabs}
\usepackage{makecell}
\usepackage{tikz}
\usepackage{url}

\usepackage{wrapfig}

\newcommand{\set}[1]{\left\{#1\right\}}
\newcommand{\norm}[1]{\left\Vert #1 \right\Vert}
\newcommand{\abs}[1]{\left\vert #1 \right\vert}
\newcommand{\ra}{\rightarrow}
\newcommand{\Real}{\mathbb{R}}
\newcommand{\eps}{\varepsilon}

\newcommand{\B}{\mathcal{B}}
\newcommand{\D}{\mathcal{D}}

\renewcommand{\subset}{\subseteq}

\newcommand{\W}{\mathcal{W}}

\newcommand{\BlackBox}{\rule{1.5ex}{1.5ex}}  
\newenvironment{proof}{\par\noindent{\bf Proof\ }}{\hfill\BlackBox\\[2mm]}
\newtheorem{example}{Example} 
\newtheorem{theorem}{Theorem}
 
\newtheorem{proposition}{Proposition} 
\newtheorem{remark}{Remark}

\newtheorem{definition}{Definition}
\newtheorem{assumption}{Assumption}

\newif\ifarXiversion
\arXiversiontrue

\begin{document}

\begin{frontmatter}
\runtitle{Physics-Informed Neural Network Lyapunov Functions}  

\title{Physics-Informed Neural Network Lyapunov Functions:\\ PDE Characterization, Learning, and Verification\thanksref{footnoteinfo}} 

\thanks[footnoteinfo]{This paper was not presented at any IFAC meeting.}

\author{Jun Liu}\ead{j.liu@uwaterloo.ca}, 
\author{Yiming Meng} \ead{yiming.meng@uwaterloo.ca}, 
\author{Maxwell Fitzsimmons}\ead{mfitzsimmons@uwaterloo.ca}, and 
\author{Ruikun Zhou}\ead{ruikun.zhou@uwaterloo.ca}

\address{Department of Applied Mathematics\\
       University of Waterloo\\
       Waterloo, Ontario N2L 3G1, Canada\\}  



\begin{abstract}                          
We provide a systematic investigation of using physics-informed neural networks to compute Lyapunov functions. We encode Lyapunov conditions as a partial differential equation (PDE) and use this for training neural network Lyapunov functions. We analyze the analytical properties of the solutions to the Lyapunov and Zubov PDEs. In particular, we show that employing the Zubov equation in training neural Lyapunov functions can lead to verifiable approximate regions of attraction close to the true domain of attraction. We also examine approximation errors and the convergence of neural approximations to the unique solution of Zubov's equation. We then provide sufficient conditions for the learned neural Lyapunov functions that can be readily verified by satisfiability modulo theories (SMT) solvers, enabling formal verification of both local stability analysis and region-of-attraction estimates in the large. Through a number of nonlinear examples, ranging from low to high dimensions, we demonstrate that the proposed framework can outperform traditional sum-of-squares (SOS) Lyapunov functions obtained using semidefinite programming (SDP).
\end{abstract}

\end{frontmatter}

\section{Introduction}

Stability analysis of nonlinear dynamical systems has been a focal point of research in control and dynamical systems. In many applications, characterizing the domain of attraction for an asymptotically stable equilibrium point is crucial. For example, in power systems, understanding the domain of attraction is essential for assessing whether the system can recover to a stable equilibrium after experiencing a fault. 

Since Lyapunov's landmark paper over a century ago \cite{lyapunov1992general}, Lyapunov functions have become an instrumental tool for nonlinear stability analysis and control design. Consequently, extensive research has focused on Lyapunov functions. One key challenge is their construction, which has been addressed through both analytical \cite{haddad2008nonlinear,sepulchre2012constructive} and computational methods \cite{giesl2007construction,giesl2015review}.

Among computational methods for Lyapunov functions, sum-of-squares (SOS) techniques have garnered widespread attention \cite{papachristodoulou2002construction,papachristodoulou2005tutorial,packard2010help,tan2008stability,topcu2008local,jones2021converse}. These methods facilitate stability analysis and provide estimates of the domain of attraction \cite{topcu2008local,tan2008stability,packard2010help}. Leveraging semidefinite programming (SDP), one can extend the region of attraction by using a specific ``shape function" within the estimated region. However, selecting shape functions in a principled manner, beyond standard norm \cite{packard2010help} or quadratic functions \cite{khodadadi2014estimation}, remains elusive.

On the other hand, Zubov's theorem \cite{zubov1964methods} characterizes the domain of attraction through a partial differential equation (PDE), unlike the commonly seen Lyapunov conditions that manifest as partial differential inequalities. Using an equation allows for precise characterization of the domain of attraction. The concept of a maximal Lyapunov function \cite{vannelli1985maximal} is closely related to Zubov's method. The authors of \cite{vannelli1985maximal} also provide a computational procedure for constructing maximal Lyapunov functions using rational functions.

Thanks to the recent surge of interest in neural networks and machine learning, many authors have recently investigated the use of neural networks for computing Lyapunov functions (see, e.g., \cite{chang2019neural,grune2021overcoming,gaby2022lyapunov,abate2020formal,gaby2022lyapunov,kang2021data}, and \cite{dawson2023safe} for a recent survey). In fact, such efforts date back to as early as the 1990s \cite{long1993feedback,prokhorov1994lyapunov}. Unlike SDP-based synthesis of SOS Lyapunov functions, neural network Lyapunov functions obtained by training are not guaranteed to be Lyapunov functions. Subsequent verification is required, e.g., using satisfiability modulo theories (SMT) solvers \cite{chang2019neural,ahmed2020automated}. The use of SMT solvers for searching and refining Lyapunov functions has been explored previously \cite{kapinski2014simulation}. Counterexample-guided search of Lyapunov functions using SMT solvers is investigated in \cite{ahmed2020automated} and the associated tool \cite{abate2021fossil}, which supports both Z3 \cite{de2008z3} and dReal \cite{gao2013dreal} as verifiers. Neural Lyapunov functions with SMT verification are explored in \cite{zhou2022neural} for systems with unknown dynamics. SMT verification is often time-consuming, especially when seeking a maximal Lyapunov function \cite{liu2023towards} or dealing with high-dimensional systems. Recent work has also focused on learning neural Lyapunov functions and verifying them through optimization-based techniques, e.g., \cite{chen2021learningHybrid,chen2021learningROA,dai2021lyapunov,dai2020counter}. Such techniques usually employ (leaky) ReLU networks and use mixed integer linear/quadratic programming (MILP/MIQP) for verification.

While recent work has demonstrated the promise of using neural networks for computing Lyapunov functions, to the best knowledge of the authors, none of the work has provided a systematic investigation of using physics-informed neural networks \cite{raissi2019physics,lagaris1998artificial} for solving the Zubov equation and using these networks to provide verified regions of attractions close to the true domain of attraction. We highlight several papers particularly related to our work. In \cite{kang2021data}, a data-driven approach is proposed to approximate solutions to Zubov's equation. It is demonstrated that neural networks can effectively approximate the solutions. However, formal verification is not conducted, and Zubov's PDE is not encoded in the training of Lyapunov functions. As demonstrated by our preliminary work \cite{liu2023towards}, encoding Zubov's equation allows us to improve the verifiable regions of attraction. Furthermore, as we shall also demonstrate in numerical examples of the current paper, formal verification is indispensable, as in many examples, we are not able to verify levels nearly as close to 1, which is predicted by the theoretical result of Zubov. This is inevitable due to the approximation errors and the compromise one has to ultimately make between using low structural complexity that enables efficient verification and high expressiveness that requires wider/deeper neural networks. In contrast, the work in \cite{grune2021computing} uses an approach closer to physics-informed neural networks (PINNs) \cite{raissi2019physics,lagaris1998artificial} for approximating a solution to Zubov's equation. However, the approach in \cite{grune2021computing} is local in nature, and the Lyapunov conditions used to train the Lyapunov functions are essentially conditions for local exponential stability. Furthermore, the author stated that using Lyapunov inequalities can lead to better training results. However, in our setting, capturing the domain of attraction unavoidably requires solving PDE instead of partial differential inequalities. The work in \cite{jones2021converse}, even though focusing on SOS approaches for approximating Lyapunov functions, is closely related to our work. The partial differential inequality constraint that the authors used to optimize the polynomial Lyapunov functions takes the form of Zubov's PDEs, although not explicitly mentioned as such in the paper. The earlier work in \cite{camilli2001generalization} has been very informative in analyzing solution properties of Zubov's equation, especially with perturbations. While their analysis of existence and uniqueness of viscosity solutions heavily relies on prior work in \cite{soravia1999optimality}, our analysis is more direct and relies on standard references on viscosity solutions for first-order PDEs \cite{bardi1997optimal,crandall1983viscosity}. We particularly highlight our novel analysis of more general nonlinear transformations between the Lyapunov and Zubov equations. To the best of the authors' knowledge, this work is the first to demonstrate that neural network Lyapunov functions can provide verifiable region-of-attraction estimates that are provably close to the true domain of attraction through both theoretical analysis and thorough numerical experiments.

A preliminary version of this paper was published as a conference paper in \cite{liu2023towards}. The current paper significantly expands the theoretical analysis and numerical experiments compared to the preliminary work in \cite{liu2023towards}. More specifically, we theoretically characterize Lyapunov functions as solutions to PDEs and conduct a more systematic investigation of solutions to Lyapunov and Zubov equations. In this process, the consideration of viscosity solutions becomes essential, as smooth solutions may not exist, as demonstrated by simple examples (see Examples 1--3 in Section 3). Employing the concept of viscosity solutions, we establish the uniqueness of solutions to Lyapunov and Zubov equations and investigate approximation errors and convergence of approximate solutions to the unique solution of Zubov's equation. None of these analyses of Lyapunov and Zubov's PDEs were present in the conference version \cite{liu2023towards}. Moreover, we have significantly expanded the set of examples solved, demonstrating that the neural Zubov approach can indeed surpass standard sum-of-squares (SOS) Lyapunov functions in approximating the domain of attraction. Notably, we have extended the range of examples from simple low-dimensional polynomial systems in \cite{liu2023towards} to include both non-polynomial systems and higher-dimensional systems not present in \cite{liu2023towards}.

\textbf{Notation:} $\mathbb{R}^n$ denotes the $n$-dimensional Euclidean space; $\abs{\cdot}$ is the Euclidean norm; $\mathbb{R}$ is the set of real numbers; $C(\Omega)$ and $C^1(\Omega)$ indicate the set of real-valued continuous and continuously differentiable functions, respectively, with domain $\Omega$; $C(\Omega,I)$ denotes the set of continuous functions with domain $\Omega$ and range $I$; $Dg$ denotes the gradient or Jacobian of a function $g$; the derivative of a univariate scalar function $V$ is sometimes denoted by $V'$; the derivative with respect to time of a time-dependent function $x$ is denoted by $\dot{x}$; the derivative of a multivariate scalar function $V$ along solutions of an ordinary differential equation is also denoted by $\dot{V}$.

\section{Problem Formulation} \label{sec:prel}

Consider a continuous-time system 
\begin{equation}\label{eq:sys}
    \dot x = f(x),
\end{equation}
where $f:\,\Real^n\ra\Real^n$ is a locally Lipschitz function. The unique solution to (\ref{eq:sys}) from the initial condition $x(0)=x_0$ is denoted by $\phi(t,x_0)$ for $t\in J$, where $J$ is the maximal interval of existence for $\phi$. 

We assume that $x=0$ is a locally asymptotically equilibrium point of (\ref{eq:sys}). The \textit{domain of attraction} of the origin for (\ref{eq:sys}) is defined as
\begin{equation}
    \D: = \set{x\in\Real^n:\,\lim_{t\ra \infty}\abs{\phi(t,x)} = 0}. 
\end{equation}
We know that $\D$ is an open and connected set \cite{bhatia1967stability}. We call any forward invariant subset of $\D$ a \textit{region of attraction} (ROA). 

Lyapunov functions can not only certify the asymptotic stability of an equilibrium point, but can also provide regions of attraction. This is achieved using sub-level sets defined as
$
\mathcal{V}_c:=\set{x\in X:\, V(x)\le c}, 
$
where $c>0$ is a positive constant, and $X$ represents a domain of interest, typically the set within which a Lyapunov function is defined. We are interested in computing regions of attraction, as they not only provide a set of initial conditions with guaranteed convergence to the equilibrium point, but also ensure state constraints and safety through forward invariance. 

The goal of this paper is to provide a systematic investigation of computing Lyapunov functions using physics-informed neural networks. We review the PDE characterization of the domain of attraction through the work of Lyapunov \cite{lyapunov1992general} and Zubov \cite{zubov1964methods}, as well as using the notion of viscosity solutions \cite{crandall1983viscosity,bardi1997optimal} for Hamilton-Jacobi equations to formalize the existence and uniqueness of solutions. We then describe algorithms for learning neural Lyapunov functions through physics-informed neural networks for solving Lyapunov and Zubov PDEs. Through a list of examples, we demonstrate that physics-informed neural network Lyapunov functions can outperform sum-of-squares (SOS) Lyapunov functions in terms of tighter region-of-attraction estimates.

\section{PDE Characterization of Lyapunov Functions}
    
\subsection{Lyapunov equation}

\begin{definition}
Let $\Omega\subset\Real^n$ be a set containing the origin. A function $V:\,\Omega\ra \Real$ is said to be {positive definite} on $\Omega$ (with respect to the origin) if $V(0)=0$ and $V(x)>0$ for all $x\in \Omega\setminus\set{0}$. We say $V$ is {negative definite} on $\Omega$ if $-V$ is positive definite on $\Omega$. 
\end{definition}

In a nutshell, a Lyapunov function for (\ref{eq:sys}) is a positive definite function whose derivative along trajectories of (\ref{eq:sys}) is negative definite. 

We refer to the following PDE as a \textit{Lyapunov equation}:
\begin{equation}
    \label{eq:lyap}
    D V \cdot f = -\omega, \quad x\in \Omega,
\end{equation}
where $\Omega$ is an open set containing the origin, $f$ is the right-hand side of (\ref{eq:sys}), and $\omega$ is a positive definite function. If we can find a positive definite solution $V$ of (\ref{eq:lyap}), then $V$ is a Lyapunov function for (\ref{eq:sys}). 

The Lyapunov equation (\ref{eq:lyap}) is a linear PDE. Its solvability is intimately tied with the solution properties of the ODE (\ref{eq:sys}). In fact, (\ref{eq:sys}) is known as the characteristic ODE for (\ref{eq:lyap}). Textbook results on the method of characteristics for first-order PDEs state that a local solution to (\ref{eq:lyap}) exists, provided that compatible and non-characteristic\footnote{Intuitively, this means the boundary is not tangent to trajectories of (\ref{eq:sys}).} boundary conditions are given \cite[Chapter 3]{evans2010partial}. 

The scenario here is somewhat different. While we assume that the origin is an asymptotically stable equilibrium point of (\ref{eq:sys}), we do not wish to impose non-characteristic boundary conditions. Rather, we prefer to impose the trivial boundary condition $V(0)=0$. Furthermore, we are interested not only in local solutions but also in the solvability of (\ref{eq:lyap}) on a given domain containing the origin. To this end, we formulate the following technical result. 

We consider two technical conditions that are sufficient for solvability of (\ref{eq:lyap}). 

\begin{assumption}\label{as:omega}
The origin is an asymptotically stable equilibrium point for (\ref{eq:sys}) and $f$ is locally Lipschitz. The function $\omega:\,\Real^n\ra \Real$ is continuous and positive definite. Define 
\begin{equation}
    \label{eq:V}
V(x)=\int_0^\infty \omega(\phi(t,x))dt,\quad x\in \Real^n, 
\end{equation}
where, if the integral diverges, we let $V(x)=\infty$. The following items hold true:
\begin{enumerate}[(i)]
    \item For any $\delta>0$, there exists $c>0$ such that $\omega(x)>c$ for all $\abs{x}>\delta$. 
    \item There exists some $\rho>0$ such that the integral $V(x)$ defined by (\ref{eq:V}) converges for all $x$ such that $\abs{x}<\rho$. 
    \item For any $\eps>0$, there exists $\delta>0$ such that $\abs{x}<\delta$ implies $V(x)<\eps$. 
\end{enumerate}
\end{assumption}

Intuitively, since $\omega$ is assumed to be positive definite and continuous, condition (i) is equivalent to that $\omega$ is non-vanishing at infinity, i.e., $\liminf_{|x|\rightarrow \infty} \omega(x) > 0$. Note that $V(0)=0$ and $V$ is nonnegative because $\omega$ is positive definite. Condition (iii) essentially states that $V$ is continuous at 0, while (ii) requires $V$ to be finite in a neighborhood of the origin. Clearly,  (iii) implies (ii). 

\begin{remark}
    If $\omega$ is Lipschitz around the origin and the origin is an exponentially stable equilibrium point for (\ref{eq:sys}), then conditions (ii) and (iii) of Assumption \ref{as:omega} hold. One common choice is $\omega(x)=x^TQx$ for some positive definite matrix $Q$. When $f(x)=Ax$ and $\omega(x)=x^TQx$, a solution of (\ref{eq:lyap}) is given by $V(x)=x^TPx$, where $P$ satisfies the celebrated Lyapunov equation $PA +A^TP=-Q$. 
\end{remark}

We first examine properties of $V$ defined by (\ref{eq:V}). 

\begin{proposition}\label{prop:V}
Let Assumption \ref{as:omega} hold. The function $V:\,\Real^n\ra\Real\cup\set{\infty}$ defined by (\ref{eq:V}) satisfies the following: 
\begin{enumerate}
    \item $V(x)<\infty$ if and only if $x\in \D$;
    \item $V(x)\ra \infty$ as $x\ra y$ for some $y\in \partial \D$; 
    \item $V$ is positive definite on $\D$;
    \item $V$ is continuous on $\D$ and its right-hand derivative along the solution of (\ref{eq:sys}) satisfies 
    \begin{equation}
        \label{eq:dini_V}
        \dot{V}(x) : = \lim_{t\ra 0^+} \frac{V(\phi(t,x))-V(x)}{t} = -\omega(x)
    \end{equation}
    for all $x\in \D$. 
\end{enumerate}
\end{proposition}

A proof of Proposition \ref{prop:V} can be found in Appendix \ref{sec:proof_prop_V}. While the Lyapunov condition (\ref{eq:dini_V}) is sufficient for stability analysis, it does not provide a PDE characterization of Lyapunov functions, nor does it reveal regularity properties of Lyapunov functions as solutions to (\ref{eq:lyap}).

Next, we further examine in what sense $V$, defined by (\ref{eq:V}), satisfies the Lyapunov equation (\ref{eq:lyap}). For simplicity of analysis and to ensure greater regularity in the solutions to (\ref{eq:lyap}), we also consider the following assumption.

\begin{assumption}\label{as:f}
    The origin is exponentially stable for (\ref{eq:sys}) and $f$ is continuously differentiable. 
\end{assumption}

\begin{proposition}\label{prop:lyap}
    Let $\Omega\subset \D$ be any open set containing the origin. Let Assumption \ref{as:omega} hold. The following statements are true:
    \begin{enumerate}
        \item $V$ defined by (\ref{eq:V}) is the unique continuous solution to (\ref{eq:lyap}) on $\Omega$ in the viscosity sense satisfying $V(0)=0$. 
        \item If $\omega$ is locally Lipschitz and Assumption \ref{as:f} holds, then $V$ defined by (\ref{eq:V}) is locally Lipschitz and satisfies (\ref{eq:lyap}) almost everywhere on $\Omega$. 
        \item If $\omega\in C^1(\Real^n)$ and Assumption \ref{as:f} holds, then $V$ defined by (\ref{eq:V}) is the unique continuously differentiable solution to (\ref{eq:lyap}) on $\Omega$ with $V(0)=0$. 
    \end{enumerate}
\end{proposition}

A proof of Proposition \ref{prop:lyap} can be found in Appendix \ref{sec:proof_prop_lyap}. Clearly, Proposition \ref{prop:lyap} holds with $\Omega=\D$. We have formulated the result as is because it is relevant when solving (\ref{eq:lyap}) on a given set $\Omega\subset\D$ without knowing $\D$.  We provide a few examples to illustrate Proposition \ref{prop:lyap}. 

\begin{example}
    Consider a scalar system $\dot x = -x^3$. Define $\omega(x)=x^2$. Then $V$ defined by (\ref{eq:V}) is $V(x)=\int_0^\infty \frac{x^2}{2tx^2+1}dt$, which does not converge for any $x\ne 0$. This is because Assumption \ref{as:omega} does not hold. 
\end{example}

\begin{example}
    Consider again the scalar system $\dot x = -x^3$. Define $\omega(x)=\abs{x}^{\frac52}$. Then $V$ defined by (\ref{eq:V}) is $V(x)=\int_0^\infty \frac{\abs{x}^{\frac{5}{2}}}{(2tx^2+1)^{\frac54}}dt = 2\sqrt{\abs{x}}$, which is not locally Lipschitz at $x=0$. Note that, while Assumption \ref{as:omega} holds and $\omega$ is locally Lipschitz, Assumption \ref{as:f} does not hold. Hence, Proposition \ref{prop:lyap}(2) is not applicable. When $x\ne 0$, $V$ is differentiable and (\ref{eq:lyap}) holds in classical sense. At $x=0$, we have $D^-V(0)=\Real$ (see Appendix \ref{sec:proof_prop_lyap} for definition). For any $p\in\Real$ and $x=0$, we have $-\omega(0)-p\cdot f(0)=0$, which verifies (\ref{eq:lyap}) is satisfied at $x=0$ in viscosity sense.  
\end{example}

\begin{example}
    Consider the scalar system $\dot x = -x$. 
    Let $\omega(x)=\abs{x}$. Then $V$ defined by (\ref{eq:V}) is $V(x)=\int_0^\infty \abs{e^{-t} x}dt = \abs{x}$. Hence, both $V$ and $\omega$ are only locally Lipschitz (and not differentiable) at $0$. 
    Furthermore, if $\omega_2(x)=\abs{x}^{\frac{3}{2}}$, then 
    $V_2(x)=\int_0^\infty \abs{e^{-t}x}^{\frac{3}{2}}dt = \frac{2}{3}\abs{x}^{\frac32}$. Similarly, both $V_2$ and $\omega_2$ are continuously differentiable (and not twice continuously differentiable) at $0$. 
    This is consistent with items (2) and (3) of Proposition \ref{prop:lyap}.
\end{example}

\begin{remark}
Item (1) of Proposition \ref{prop:lyap} can be seen as a special case of the results stated in \cite{bardi1997optimal,camilli2001generalization}. Here, we provide a more direct proof. Local Lipschitz continuity of $V$ (in a more general setting with perturbation) has been established in \cite{camilli2001generalization}, but under more restrictive conditions. Item (3) is perhaps not surprising, yet we are not aware of a similar result being stated in the literature.
\end{remark}

\subsection{Zubov equation}

While Proposition \ref{prop:lyap} characterizes the Lyapunov function $V$ within the domain of attraction using Lyapunov's PDE (\ref{eq:pde}), due to Proposition \ref{prop:V}(2), any finite sublevel set of $V$ does not yield the domain of attraction. Zubov's theorem is a well-known result that states the sublevel-1 set of a certain Lyapunov function is equal to the domain of attraction. We state Zubov's theorem below. 

\begin{theorem}[Zubov's theorem \cite{zubov1964methods}]\label{thm:zubov}
    Let $\Omega\subset\Real^n$ be an open set containing the origin. Then $\Omega=\D$ if and only if there exists two continuous functions $W:\,\Omega\ra \Real$ and $\Psi:\,\Omega\ra \Real$ such that the following conditions hold:
 \begin{enumerate}
     \item $0<W(x)<1$ for all $x\in \Omega\setminus\set{0}$ and $W(0)=0$; 
     \item $\Psi$ is positive definite on $\Omega$ with respect to the origin;  
     \item for any sufficiently small $c_3>0$, there exist two positive real numbers $c_1$ and $c_2$ such that $\abs{x}\ge c_3$ implies $W(x)>c_1$ and $\Psi(x)>c_2$; 
     \item $W(x)\ra 1$ as $x\ra y$ for any $y\in \partial \Omega$;
     \item $W$ and $\Psi$ satisfy
     \begin{equation}\label{DV:zubov}
     \dot W (x) = -\Psi(x)(1-W(x)), 
     \end{equation}
     where $\dot W$ is the right-hand derivative of $W$ along solutions of (\ref{eq:sys}) as defined in (\ref{eq:dini_V}).  
 \end{enumerate}
\end{theorem}

We first highlight a connection between $V$ defined by (\ref{eq:V}), with properties listed in Proposition \ref{prop:V}, and the function $W$ in Zubov's theorem.  Let $\beta:\,\Real\ra\Real$ satisfy 
\begin{equation}\label{eq:beta}
    \dot \beta =  (1-\beta)\psi(\beta),\quad \beta(0)=0,
\end{equation}
where $\psi$ is a locally Lipschitz function satisfying $\psi(s)>0$ for $s\ge 0$. Clearly, any function satisfying (\ref{eq:beta}) is continuously differentiable, strictly increasing on $[0,\infty)$, and satisfies $\beta(0)=0$ and  $\beta(s)\ra 1$ as $s\ra \infty$. 

With $V$ defined by (\ref{eq:V}), let 
\begin{equation}
    \label{eq:W}
W(x) = \left\{\begin{aligned}
&\beta(V(x)),\text{ if } V(x)<\infty,\\
&1, \quad  \text{otherwise},
\end{aligned}\right.
\end{equation}
where $\beta:\,[0,\infty)\ra\Real$ satisfies (\ref{eq:beta}). 

We can verify the following properties for $W$, a proof of which can be found in \cite{liu2023towards}. 

\begin{proposition}\label{prop:W}
The function $W:\,\Real^n\ra\Real$ defined by (\ref{eq:W}) is continuous and satisfies the conditions in Theorem \ref{thm:zubov} on $\D$ with $\Psi(x) = \psi(\beta(V(x)))\omega(x)$. 
\end{proposition}

Motivated by this, we consider a slightly generalized Zubov equation as follows 
\begin{equation}
    \label{eq:zubov}
    DW \cdot f = - \omega \psi(W)(1-W),  
\end{equation}
for $x\in \Omega$, where $\Omega\subset\Real^n$ is an open set on which we are interested in solving (\ref{eq:zubov}). 

Next, we characterize solutions of Zubov's equation (\ref{eq:zubov}) for the case where $\Omega$ is bounded and a boundary condition on $\partial\Omega$ is specified as follows 
\begin{equation}
    \label{eq:zubov_boundary}
    W = g,
\end{equation}
where $g$ is taken to be consistent with (\ref{eq:W}), i.e., $g(x)=1$ for $x\in \partial\Omega\setminus\D$ and $g(x)=\beta(V(x))$ for $x\in \partial\Omega\cap \D$.
We put forward a technical assumption on $\psi$. 

\begin{assumption}\label{as:psi}
There exists a nonempty interval $I$ such that the function $G:\,I\ra\Real$ defined by $s\mapsto (1-s)\psi(s)$ is monotonically decreasing on $I$. 
\end{assumption}

Assumption \(\ref{as:psi}\) on \(\psi\) stipulates a nonempty interval \(I\), which depends on \(\psi\) and serves as the range of the solution \(W\) to (\ref{eq:zubov}), as required by the following result.

\begin{theorem}\label{thm:zubov_pde}
    Let Assumptions \ref{as:omega} and \ref{as:psi} hold. Let $\Omega\subset\Real^n$ be a bounded open set containing the origin. Suppose that $\omega$ is locally Lipschitz. The following statements are true:
    \begin{enumerate}
        \item $W$ defined by (\ref{eq:W}) is the unique viscosity solution to (\ref{eq:zubov}) in  $C(\bar\Omega,I)$ satisfying $W(0)=0$ and the boundary condition (\ref{eq:zubov_boundary}) on $\partial \Omega$.        
        \item If Assumption \ref{as:f} holds, then $W$ defined by (\ref{eq:W}) is locally Lipschitz on  $\Real^n\setminus \partial \D$ and satisfies (\ref{eq:zubov}) almost everywhere on $\Real^n\setminus \partial \D$. 
        \item If $\omega\in C^1(\Real^n)$ and Assumption \ref{as:f} holds, then $W$ defined by (\ref{eq:W}) is in $C^1(\Real^n\setminus \partial \D)$. 
    \end{enumerate}    
\end{theorem}

A proof of Theorem \ref{thm:zubov_pde} can be found in Appendix \ref{sec:proof_prop_zubov}. 
An interesting byproduct of the analysis in the proof of Theorem \ref{thm:zubov_pde} is the following error estimate for a neural (or other) approximate solution to the PDE (\ref{eq:zubov}). To derive the result, we refer to the following technical condition on an approximate solution $v:\,\bar\Omega\ra I$. Let $G$ and $I$ be from Assumption \ref{as:psi}. For some $\eps>0$ and $\delta>0$, we have 
    \begin{equation}
       \label{eq:v_eps}
   G(v(x)) + [-\frac{\eps}{\min_{x\in\bar\Omega\setminus\B_{\delta}}\omega(x)}, \frac{\eps}{\min_{x\in\bar\Omega\setminus\B_{\delta}}\omega(x)}]\subset G(I)     
    \end{equation}
    for all $x\in \Omega\setminus\bar\B_\delta$. We shall see that, for any fixed $\delta>0$,  condition (\ref{eq:v_eps}) holds for $\eps>0$ sufficiently small, provided that a mild condition holds for $v$ (see Remark \ref{rem:psi}).

\begin{proposition}\label{prop:error}
    Suppose that the assumptions of Theorem \ref{thm:zubov_pde} hold. Fix any $\delta>0$ such that $\B_\delta\subset\Omega\cap\D$. 
    Let $W$ be the unique viscosity solution to (\ref{eq:zubov}) with boundary condition (\ref{eq:zubov_boundary}). The following statements hold:
    \begin{enumerate}
        \item   (\textbf{error estimate})  For any $\eps>0$, let $v$ be an $\eps$-approximate viscosity solution (see its definition in Appendix \ref{sec:proof_prop_error}) to (\ref{eq:zubov}) on $\Omega\setminus\bar\B_\delta$ with boundary error $\abs{W(x)-v(x)}\le \eps_b$ on $\partial\Omega \cup \partial \bar\B_{\delta}$. Assume that $v$ and $\eps$ satisfy condition (\ref{eq:v_eps}).
    Then there exists a constant $C(\eps,\eps_b)$, satisfying $C(\eps,\eps_b)\ra 0$ as $\eps\ra 0$ and $\eps_b\ra 0$, such that $\abs{W(x)-v(x)}\le C(\eps,\eps_b)$ for all $x\in\bar\Omega\setminus\B_\delta$. 

    \item (\textbf{convergence}) Furthermore, suppose that $\set{v_n}$ is a sequence of $\eps_n$-approximate viscosity solutions to (\ref{eq:zubov}) on $\Omega$ with a uniform Lipschitz constant\footnote{The same statement holds if they share a uniform modulus of continuity. 
    } on $\bar\Omega$, where $\eps_n\downarrow 0$, and with boundary condition $v_n(0)\ra 0$ and $v_n(x)\ra W(x)$ uniformly on $\partial \Omega$ as $n\ra \infty$. Assume that, for each $\delta>0$, there exists some $N>0$ such that (\ref{eq:v_eps}) holds for all $(v_n,\eps_n)$ with $n>N$. Then $\set{v_n}$ converges to $W$ uniformly on $\bar\Omega$.  
    \end{enumerate}
\end{proposition}

\begin{remark}\label{rem:psi}
    Two special cases of $\psi(s)$ are given by (i) $\psi(s)=\alpha$ or (ii) $\psi(s)=\alpha(1+s)$ for some constant $\alpha>0$, which correspond to $\beta(s)=1-\exp(-\alpha s)$ and $\beta(s)=\tanh(\alpha s)$,  respectively. Such transforms are used in, e.g., \cite{camilli2001generalization,kang2021data,meng2023learning}. 
    For case (i), Assumption \ref{as:psi} holds with $I=\Real$. It is clear that in this case $G(I)=\Real$ and (\ref{eq:v_eps}) trivially holds. For case (ii), we can take $I=[0,\infty)$. Then $G(I)=(-\infty,\alpha]$. Condition  (\ref{eq:v_eps}) for $v$ and $\eps$ in Proposition \ref{prop:error} holds if there exists some $c>0$ such that $v(x)>c$, whenever $\abs{x}>\delta$, and $\eps>0$ sufficiently small such that $G(c)+\frac{\eps}{\min_{x\in\bar\Omega\setminus\B_{\delta}}\omega(x)}\le \alpha$. It further follows that (\ref{eq:v_eps}) holds for $(v_n,\eps_n)$ in Proposition \ref{prop:error}, if for each $\delta>0$, there exists $c>0$ and $N>0$ such that $v_n(x)>c$ for all $\abs{x}>\delta$ and all $n>N$. 
\end{remark}

\begin{example}
    Consider the scalar system
    $\dot x = -x + x^3$. It has three equilibrium points at $\set{0,\pm 1}$. The origin is exponentially stable with domain of attraction $\D=(-1,1)$. Consider 
    $
    V(x) = \int_0^\infty \abs{\phi(t,x)}^{2}dt. 
    $
    By Proposition \ref{prop:V}, we have
    $
    \dot V(x) = V'(x)(-x+x^3) = - x^2. 
    $
    For $x\in (0,1)$, assuming differentiability of $V$, we have
    $
    V'(x) = \frac{x}{1-x^2}. 
    $
    Integrating this with the condition $V(0)=0$ gives 
    $
    V(x) = -\frac12 \ln(1-x^2). 
    $
    Taking $W(x)=1-\exp(-\alpha V(x))$ (corresponding to $\psi(x)=\alpha>0$ as indicated in Remark \ref{rem:psi}), we obtain
    $
    W(x) = 1 - (1-x^2)^{\frac{\alpha}{2}}.   
    $
    One can easily verify that $W$ satisfies conditions in Zubov's theorem. It can also be easily verified that $W$ can fail to be differentiable at $x\in\set{\pm 1} = \partial \D$, or even locally Lipschitz for $\alpha<2$. 
\end{example}

To conclude this section, we highlight the crucial connections of the theoretical analysis presented here with the training and verification of Lyapunov functions to be discussed in Sections \ref{sec:algo} and \ref{sec:verify}. The main result (Theorem \ref{thm:zubov_pde}) shows that solving Zubov's PDE (\ref{eq:zubov}) with boundary condition (\ref{eq:zubov_boundary}) can lead to a Lyapunov function \(W\), whose sublevel-1 set gives the true domain of attraction \(\D = \{x \in \Real^n : W(x) < 1\}\). This analysis fundamentally relies on the analysis of solutions to Lyapunov's PDE (\ref{eq:lyap}) (Proposition \ref{prop:lyap}). Building upon these theoretical results, subsequent sections will focus on computing approximate solutions by solving Zubov's PDE (\ref{eq:zubov}) using neural networks and verifying them with SMT solvers. To this end, Proposition \ref{prop:error} also plays an instrumental role in ensuring that solving Zubov's PDE (\ref{eq:zubov}) with small residual errors (\(\eps\)-approximate solutions) indeed leads to accurate approximations of the true solution.

\section{Physics-Informed Neural Lyapunov Function}\label{sec:algo}

A physics-informed neural network (PINN) \cite{lagaris1998artificial,raissi2019physics} is essentially a neural network that solves a PDE. In this section, we present the main algorithm for solving Lyapunov and Zubov equations. To describe the algorithm,  consider a first-order PDE of the form
\begin{equation}
    \label{eq:1st_order_PDE}
    F(x,W,DW) = 0,\quad x\in \Omega,
\end{equation}
subject to the boundary condition $W = g$ on $\partial\Omega$.

Consider a general multi-layer feedforward neural network function $W_N(x;\theta)$ defined inductively as follows. 
Let the output of the first layer (input layer) be $a^{(0)} = x$, where $x$ is the input vector. For each subsequent layer $l$ (where $1 \leq l \leq L$), the output is defined inductively as 
$
a^{(l)} = \sigma^{(l)}(H^{(l)} a^{(l-1)} + b^{(l)}), 
$
where $H^{(l)}$ and $b^{(l)}$ are the weight matrix and bias vector for layer $l$, and $\sigma^{(l)}$ is the activation function for layer $l$. The final output of the network is $y = a^{(L)}$. The parameter vector $\theta$ consists of all weights $H^{(l)}$ and biases $b^{(l)}$.

We train $W_N(x;\theta)$ as a PINN Lyapunov function by optimizing $\theta$ with respect to a specifically designed loss function that encodes the PDE (\ref{eq:1st_order_PDE}), along with the correct boundary condition and additional data loss. We describe the algorithm as follows. 
\begin{enumerate}
    \item Choose a set of interior collocation points $\set{x_i}_{i=1}^{N_c}\subset\Omega$. These are points at which the test solution $W_N(x;\beta)$ and its derivative will be evaluated to obtain the mean-square residual error $\frac{1}{N_c}\sum_{i=1}^{N_c} F(x_i,W_N(x_i;\beta),DW_N(x_i;\beta))^2$ for (\ref{eq:1st_order_PDE}) at these points.

    \item Choose a set of boundary points $\set{y_i}_{i=1}^{N_b}\subset\partial\Omega$ at which the mean-square boundary error $\frac{1}{N_b}\sum_{i=1}^{N_b}(W_N(y_i;\beta)-g(y_i))^2$ can be evaluated. 
    
    \item In some cases, we may also obtain a set of data points $\{(z_i,\hat{W}(z_i))\}_{i=1}^{N_d}$, where $\set{z_i}_{i=1}^{N_d}\subset \Omega$ and $\{\hat W(z_i)\}_{i=1}^{N_d}$ are approximations to the ground truth values of $W$ at $\set{z_i}_{i=1}^{N_d}$. For examples, this may correspond to evaluation of $V$ defined in (\ref{eq:V}) or $W$ defined in (\ref{eq:W}) through numerical integration of (\ref{eq:sys}). The mean-square data loss can be defined as 
    $\frac{1}{N_d}\sum_{i=1}^{N_d}(W_N(z_i;\beta)-\hat W(z_i))^2$. 
\end{enumerate}

The loss function for optimizing $\theta$ is given by
\begin{align}
    \text{Loss}(\theta) &= \frac{1}{N_c}\sum_{i=1}^{N_c} F(x_i,W_N(x_i;\theta),DW_N(x_i;\theta))^2 \notag \\
    &\quad + \lambda_b \frac{1}{N_b} \sum_{i=1}^{N_b}(W_N(y_i;\theta)-g(y_i))^2,  \notag  \\
    &\quad + \lambda_d \frac{1}{N_d} \sum_{i=1}^{N_d}(W_N(z_i;\theta)-\hat W(z_i))^2,   
    \label{eq:lossV}
\end{align}
where $\lambda_b>0$ and $\lambda_d>0$ are weight parameters. Training of $\theta$ can be achieved with standard gradient descent. 


\begin{remark}
    In formulating the PDE (\ref{eq:1st_order_PDE}), Lyapunov equation has a single boundary condition $V(0)=0$, whereas the boundary condition for Zubov equation consists of both $W(0)=0$ and $W(y)=1$ for $y\not\in\D$, if $\D\subset\Omega$. For examples where $\D$ is not a  subset of $\Omega$, the boundary condition may also consist of $W(y)=\beta(V(y))$ for $y\in\D\cap\partial\Omega$; see comment after (\ref{eq:zubov_boundary}). In Section \ref{sec:examples}, the Van der Pol equation (Example \ref{ex:vdp}) belongs to the former case, and the two-machine power system (Example \ref{ex:power}) belongs to the latter case.
\end{remark}

\begin{remark}
A key feature of the proposed approach, compared to state-of-the-art sum-of-squares methods for estimating regions of attraction, is that \(\text{Loss}(\theta)\) is a non-convex function of \(\theta\). While solving non-convex optimization problems is more challenging\textemdash since obtaining a global minimum is not guaranteed and convergence rates may be slower\textemdash embracing non-convex optimization allows us to utilize the representational power of deeper neural networks. This is especially advantageous when computing regions closer to the true domain of attraction, as demonstrated by the numerical examples in Section \ref{sec:examples}. 
\end{remark}

\section{Formal Verification}\label{sec:verify}

Neural network functions, as computed by solving the Lyapunov equation (\ref{eq:lyap}) and the Zubov equation (\ref{eq:zubov}) using algorithms outlined in Section \ref{sec:algo}, offer no formal stability guarantees. This is because neural approximations come with numerical errors, and training deep neural networks usually does not guarantee convergence to a global minimum. On the other hand, Lyapunov functions are sought precisely because they provide formal guarantees of stability, especially in safety-critical applications. For this reason, formal verification of neural Lyapunov functions is indispensable if formal stability guarantees and rigorous regions of attraction are required. We outline such verification procedures in this section.

\subsection{Verification of Local Stability}
\label{sec:local}

In this section, we assume that Assumption \ref{as:f} holds, i.e., $f$ is continuously differentiable and $x=0$ is an exponentially stable equilibrium point of (\ref{eq:sys}). Rewrite (\ref{eq:sys}) as 
\begin{equation}
    \label{eq:nonlinear}
    \dot x = Ax + g(x),
\end{equation}
where $g(x)=f(x)-Ax$ satisfies $\lim_{x\ra 0}\frac{\norm{g(x)}}{\norm{x}}=0$. Given any symmetric positive definite matrix $Q\in\Real^{n\times n}$, let $P$ be the unique solution to the Lyapunov equation 
\begin{equation}
    \label{eq:lyap_linear}
    PA +A^T P = -Q.
\end{equation}
Let $V_P(x)=x^TPx$. Then 
\begin{align*}
\dot{V}_P(x) 
&= - x^TQx + 2x^TPg(x) \\
& \le - \lambda_{\min}(Q)\norm{x}^2 + 2x^TPg(x) \\
& = -\eps \norm{x}^2  + (2x^TPg(x) - r\norm{x}^2), 
\end{align*} 
where we set $r=\lambda_{\min}(Q)-\eps>0$ for some sufficiently small $\eps>0$ and $\lambda_{\min}(Q)$ is the minimum eigenvalue of $Q$. If we can determine a constant $c>0$ such that 
\begin{equation}\label{eq:c1_P}
V_P(x)\le c \Longrightarrow 2x^TPg(x)\le r\norm{x}^2, 
\end{equation}
then we have verified 
\begin{equation}\label{eq:stability}
V_P(x)\le c \Longrightarrow \dot{V}_P(x)\le -\eps\norm{x}^2,
\end{equation}
which verifies exponential stability of the origin, and a region of attraction $\set{x\in\Real^n:\,V_P(x)\le c}$. 


While all this aligns with standard quadratic Lyapunov analysis, we note a subtle issue in the verification of (\ref{eq:c1_P}) using SMT solvers. For example, although Z3 \cite{de2008z3} offers exact verification, unlike dReal's numerical verification \cite{gao2013dreal}, it currently lacks support for functions beyond polynomials. For non-polynomial vector fields, we found dReal to be the state-of-the-art verifier. However, due to the conservative use of interval analysis in dReal to account for numerical errors, verification of inequalities such as $2x^TPg(x)\le r\norm{x}^2$ may return a counterexample close to the origin. To overcome this issue, we consider a higher-order approximation of $Pg(x)$. By the mean value theorem, we have
$$
    P g(x) = P g(x) - P g(0) = \int_0^1 P\cdot Dg(tx)dt\cdot x,
$$
where $Dg$ is the Jacobian of $g$ given by $Dg= Df - A$, 
which implies
$$
 2x^TP g(x) \le 2\sup_{0\le t\le 1}\norm{P \cdot Dg(tx)}\norm{x}^2.
$$
As a result, to verify (\ref{eq:c1_P}), we just need to verify 
\begin{align}
V_P(x)\le c &\Longrightarrow 2\sup_{0\le t\le 1}\norm{P\cdot Dg(tx)}\le r.\label{eq:cr2}
\end{align}
By convexity of the set $\set{x\in\Real^n:\,V_P(x)\le c}$, (\ref{eq:cr2}) is equivalent to 
\begin{align}
V_P(x)\le c &\Longrightarrow 2\norm{P\cdot Dg(x)}\le r.\label{eq:cr3}
\end{align}
Since $Dg(0)=0$ and $Dg$ is continuous, one can always choose $c>0$ sufficiently small such that (\ref{eq:cr3}) can be verified. Furthermore, in rare situations, if $2\norm{P\cdot Dg(x)}\le r$ can be verified for all $x\in \Real^n$, then global exponential stability of the origin is proved for (\ref{eq:sys}).  

\begin{remark}
From (\ref{eq:cr3}), one can further use easily computable norms, such as the Frobenius norm, to over-approximate the matrix 2-norm $\norm{P\cdot Dg(x)}$.
\end{remark}

\begin{example}
    Consider an inverted pendulum 
    \begin{equation}
\begin{aligned}
        \dot{x}_1 & = x_2, \\
        \dot{x}_2 & = \sin(x_1) - x_2 - (k_1 x_1 + k_2 x_2),
\end{aligned}\label{eq:ip}
    \end{equation}
where the linear gains are given by 
$k_1=4.4142, k_2= 2.3163$. 
In the literature, there have been numerous references to this example as a benchmark for comparing techniques for stabilization with provable ROA estimates. Interestingly, with $Q=I$ and $\epsilon=10^{-4}$, which leads to $r=0.9999$, it turns out that one can verify $2\|P\cdot Dg(x)\|\leq r$ with dReal \cite{gao2013dreal} within 1 millisecond, which confirms global stability of the origin for (\ref{eq:ip}). 
\end{example}

\begin{remark}
    While stability analysis using quadratic Lyapunov functions is certainly standard, we highlight that our approach to verifying (\ref{eq:c1_P}) via the mean value theorem and (\ref{eq:cr3}) successfully overcomes the limitations in current approaches that exclude verification around a small neighbourhood of the origin to avoid the numerical conservativeness of SMT solvers \cite{chang2019neural,zhou2022neural}. 
\end{remark}

\subsection{Verification of Regions of Attraction}\label{sec:roa}

We outline how local stability analysis via linearization and reachability analysis via a Lyapunov function can be combined to provide a region-of-attraction estimate that is close to the domain of attraction.  Let $V_P$ be a quadratic Lyapunov function as defined in Section \ref{sec:local} and $W_N$ be a neural Lyapunov function learned using algorithms outlined in Section \ref{sec:algo}. Suppose that we have verified a local region of attraction $\set{x\in \Real^n:\,V_P(x)\le c}$ for some $c>0$.  Let $X\subset\Real^n$ denote a compact set on which verification takes place. We can verify the following inequalities using SMT solvers: 
\begin{align}
(c_1\le {W}_N(x) \le c_2) \wedge (x\in X) &\Longrightarrow \dot{W}_N(x)  \le -\eps,\label{eq:dW}\\
({W}_N(x)\le c_1) \wedge (x\in X)  &\Longrightarrow V_P(x)\le c,\label{eq:WP}
\end{align}
where $\eps>0$, $c_2>c_1>0$. 

\begin{proposition}
    If (\ref{eq:dW}) and (\ref{eq:WP}) hold and the set $\W_{c_2}=\set{x\in X:\, W_N(x)\le c_2}$ does not intersect with the boundary of $X$, then $\W_{c_2}$ is a region of attraction for (\ref{eq:sys}). 
\end{proposition}

\begin{proof}
A solution starting from $\W_{c_2}$ remains in $\W_{c_2}$ as long as the solution do not leave $X$. However, to leave $X$ it has to cross the boundary of $\W_{c_2}$ first. This is impossible because of (\ref{eq:dW}). Within $\W_{c_2}$, solutions converge to $\set{x\in\Real^n:\,x^TPx\le c}$ in finite time, which is a verified region of attraction.  
\end{proof}
If $W_N$ is obtained by solving the Zubov equation (\ref{eq:zubov}) using algorithms proposed in Section \ref{sec:algo}, then the set $\W_{c_2}$ can provide a region of attraction close to the true domain of attraction. We shall demonstrate this in the next section using numerical examples, where it is shown that the neural Zubov approach outperforms sum-of-squares Lyapunov functions in this regard.

\section{Numerical Examples}\label{sec:examples}

In this section, we present numerical examples demonstrating our ability to effectively compute neural Lyapunov functions by solving Lyapunov and Zubov equations. We show that the obtained neural Lyapunov functions can be formally verified, yielding certified regions of attraction. Thanks to the theoretical guarantees of Zubov's equation and the expressiveness of neural networks, these regions of attraction can outperform those derived from sum-of-squares Lyapunov functions obtained using semidefinite programming.  

\textbf{Implementation details}: Numerical experiments were conducted on a 2020 MacBook Pro with a 2 GHz Quad-Core Intel Core i5, without any GPU. For training, we randomly selected 300,000 collocation points in the domain and trained for 20 epochs using the Adam optimizer with PyTorch. The sizes of the neural networks are as reported in Tables \ref{tab:vdp} and \ref{tab:power}. We used the hyperbolic tangent as the activation function. Verification was done with the SMT solver dReal \cite{gao2013dreal}. All examples are solved using the tool LyZNet \cite{liu2024lyznet}. The code is available at \url{https://git.uwaterloo.ca/hybrid-systems-lab/lyznet}. Since formal verification is conducted after training, we choose neural network sizes that can be effectively verified by the SMT solver dReal \cite{gao2013dreal}. Our numerical studies indicate that a network with two hidden layers of up to 30 neurons each, or three hidden layers of up to 10 neurons, provides a good balance between expressiveness and effective verifiability. The numerical examples below, unless otherwise noted, all use neural networks with two hidden layers of 30 neurons each. Table \ref{tab:time} summarizes the computational times for training and verification. Note that computational time largely depends on computer hardware. The computational times provided in Table \ref{tab:time} serve as a reference point.

\begin{table}[h!]
  \caption{Training and verification times for numerical examples}
  \label{tab:time}
  \begin{tabular}{ccc}
    \toprule
Model &  Train (sec) & Verify (sec) \\ 
    \midrule
Van der Pol ($\mu=1$) & 494 & 973\\
Van der Pol ($\mu=3$) & 502 & 442\\
two-machine power & 473 & 3,375  \\
10-d system & 92 & 356 \\
20-d networked Van der Pol &  27,342 & 46,611\\
\bottomrule
\end{tabular}
\end{table}

\begin{example}[Van der Pol equation]\label{ex:vdp}
Consider the reversed Van der Pol equation given by
    \begin{equation}
\begin{aligned}
        \dot{x}_1 & = -x_2, \\
        \dot{x}_2 & = x_1 - \mu (1 - x_1^2) x_2,
\end{aligned}
    \end{equation}
where $\mu > 0$ is a stiffness parameter.
\end{example}

For $\mu = 1.0$ and $X = [-2.5, 2.5] \times [-3.5, 3.5]$, we train a neural network to solve Zubov's PDE (\ref{eq:zubov}) on $X$. Figure \ref{fig:vdp1} depicts the largest verifiable sublevel set, along with the learned neural Lyapunov function. It can be seen that the verified ROA is quite comparable and slightly better than that provided by an SOS Lyapunov function with a polynomial degree of six, obtained using a standard ``interior expanding" algorithm \cite{packard2010help}. As the stiffness of the equation increases, we observe further improved advantages of neural Lyapunov approach over SOS Lyapunov functions. With $\mu=3.0$ and domain $X=[-3.0, 3.0]\times [-6.0, 6.0]$, the comparison is shown in Figure \ref{fig:vdp3}.  

\begin{figure}[h]
  \centering
  \includegraphics[width=\linewidth]{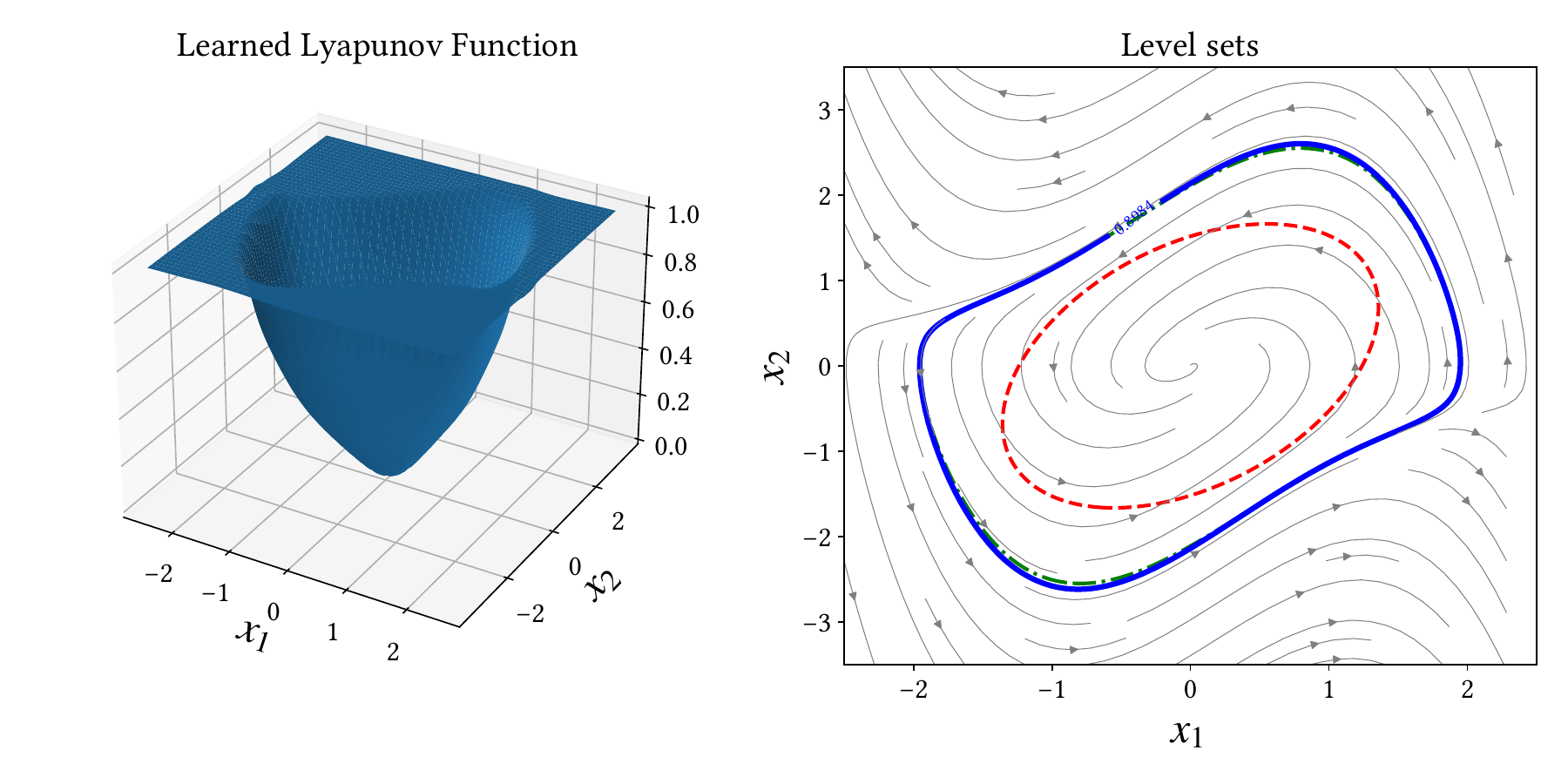}
  \caption{Verified neural Lyapunov function for Van der Pol equation with $\mu=1.0$ (Example \ref{ex:vdp}). Dashed red: quadratic Lyapunov function; dot-dashed green: SOS Lyapunov; solid blue: neural Lyapunov.}
  \label{fig:vdp1}
\end{figure}

\begin{figure}[h]
  \centering
  \includegraphics[width=\linewidth]{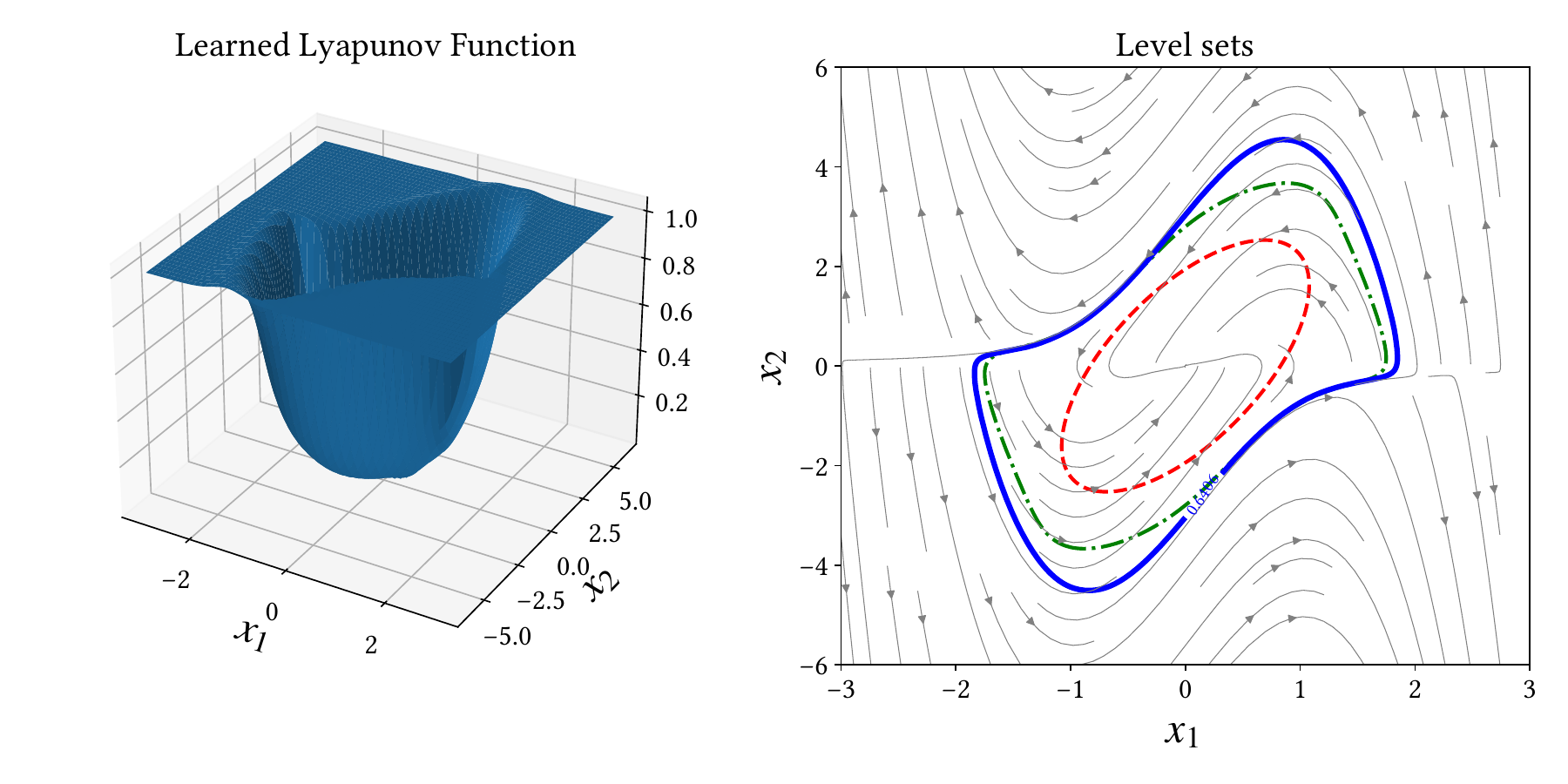}
  \caption{Verified neural Lyapunov function for Van der Pol equation with $\mu=3.0$ (Example \ref{ex:vdp}). Dashed red: quadratic Lyapunov function; dot-dashed green: SOS Lyapunov; solid blue: neural Lyapunov. The learned neural Lyapunov function outperforms a sixth degree SOS Lyapunov function.}
  \label{fig:vdp3}
\end{figure}

\begin{table}
  \caption{Parameters and verification results for Van der Pol equation (Example \ref{ex:vdp})}
  \label{tab:vdp}
  \begin{tabular}{cccccc}
    \toprule
 $\mu$  & Layer & Width & $c_2$ & Volume 
 & SOS volume \\ 
    \midrule
    1.0 & 2 & 30 & 0.898 & \textbf{95.64\%} &  94.17\% \\
    \midrule
    3.0 & 2 & 30 & 0.640  & \textbf{85.12\%} &  70.88\%\\
    \bottomrule
  \end{tabular}
\end{table}

\begin{example}[Two-machine power system]\label{ex:power}Consider the two-machine power system \cite{vannelli1985maximal} modelled by 
    \begin{equation}
\begin{aligned}
        \dot{x}_1 & = x_2, \\
        \dot{x}_2 & = -0.5x_2 - (\sin(x_1 +\delta)-\sin(\delta)),
\end{aligned}
    \end{equation}
    where $\delta = \frac{\pi}{3}$. Note that the system has an unstable equilibrium point at $(\pi/3,0)$.   
\end{example}

    Figure \ref{fig:power} shows that a neural network with two hidden layers and 30 neurons in each layer provides a region-of-attraction estimate significantly better than that from a sixth-degree polynomial SOS Lyapunov function, computed with a Taylor expansion of the system model. The example shows that neural Lyapunov functions perform better than SOS Lyapunov functions when the nonlinearity is non-polynomial. We also compared with the rational Lyapunov function presented in \cite{vannelli1985maximal}, but the ROA estimate is worse than that from the SOS Lyapunov function, and we were not able to formally verify the sublevel set of the rational Lyapunov function reported in \cite{vannelli1985maximal} with dReal \cite{gao2013dreal}. Improving the degree of polynomial in the SOS approach does not seem to improve the result either. 

\begin{figure}[h]
  \centering
  \includegraphics[width=\linewidth]{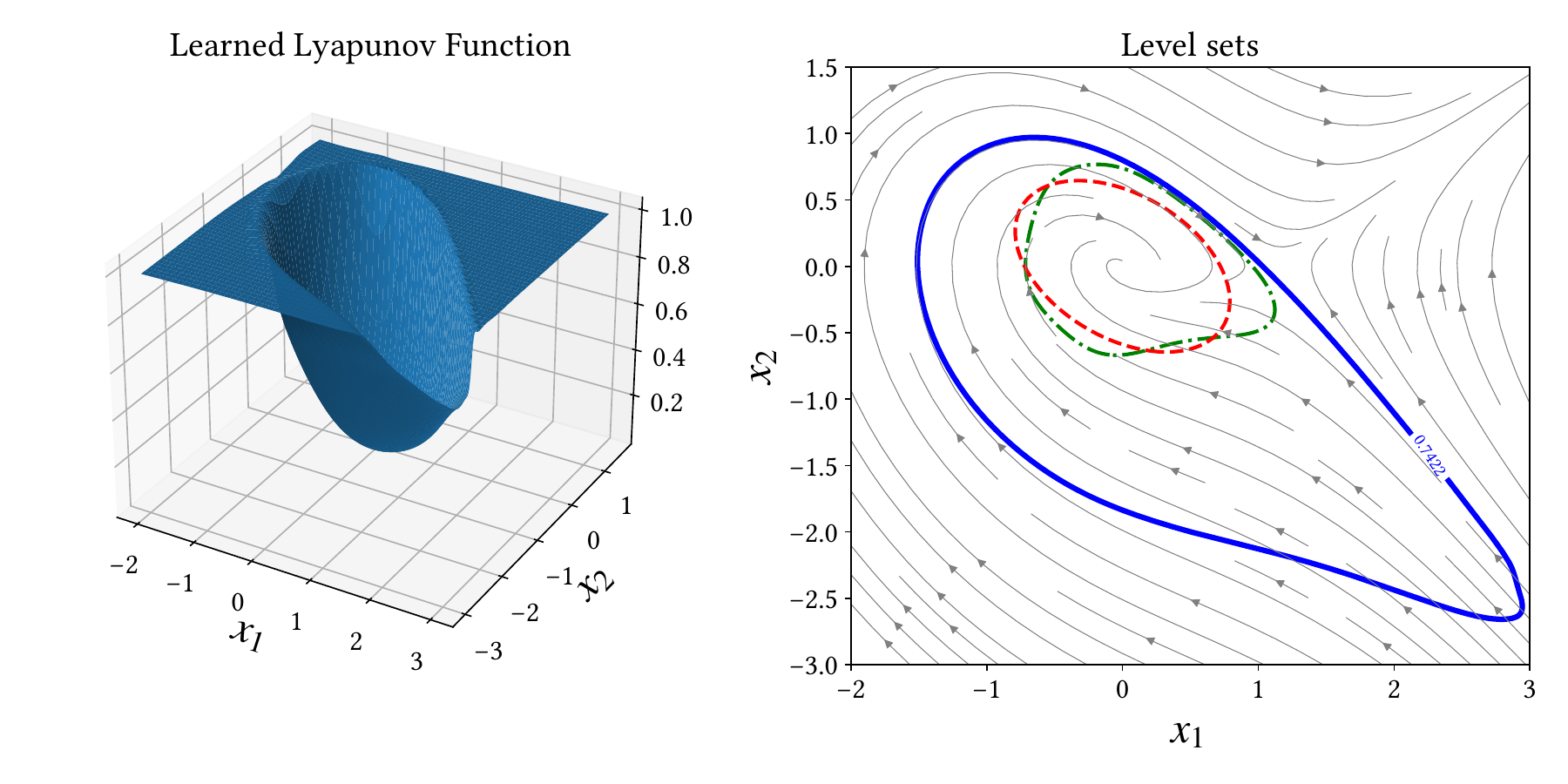}
  \caption{Verified neural Lyapunov function for a two-machine power system (Example \ref{ex:power}). Dashed red: quadratic Lyapunov function; dot-dashed green: SOS Lyapunov; solid blue: neural Lyapunov. The learned neural Lyapunov function significantly outperforms a sixth degree SOS Lyapunov function.} 
  \label{fig:power}
\end{figure}

\begin{table}
  \caption{Parameters and verification results for a two-machine power system (Example \ref{ex:power})}
  \label{tab:power}
  \begin{tabular}{ccccc}
    \toprule
  Layer & Width & $c_2$ & Volume  
 & SOS volume\\ 
    \midrule
 2 & 30 & 0.742 & \textbf{82.52\%} &  18.53\% \\
    \bottomrule
  \end{tabular}
\end{table}

\begin{example}[10-dimensional system] \label{ex:10d}
    Consider the 10-dimensional nonlinear system from \cite{grune2021computing}:
\begin{align*}
 \dot x_1 =    &-x_1 + 0.5 x_2 - 0.1 x_9^2, \\
 \dot x_2 =   &-0.5 x_1 - x_2, \\
 \dot x_3 =    &-x_3 + 0.5 x_4 - 0.1 x_1^2, \\
 \dot x_4 =    &-0.5 x_3 - x_4, \\
 \dot x_5 =  &-x_5 + 0.5 x_6 + 0.1 x_7^2, \\
 \dot x_6 =   &-0.5 x_5 - x_6, \\
 \dot x_7 =   &-x_7 + 0.5 x_8, \\
 \dot x_8 =   &-0.5 x_7 - x_8, \\
 \dot x_9 =   &-x_9 + 0.5 x_{10}, \\
 \dot x_{10} =    &-0.5 x_9 - x_{10} + 0.1 x_2^2.
\end{align*}
\end{example}
The example was used by the authors of \cite{grune2021computing} and \cite{gaby2022lyapunov} to illustrate the training of neural network Lyapunov functions for stability analysis, where the trained Lyapunov functions were not verified. We remark that global asymptotic stability of this system can be easily established using an input-to-state stability argument. Therefore, when training neural Lyapunov functions for this example, one is essentially computing a local Lyapunov function within the domain of attraction. 

Using the verification method from Section \ref{sec:verify}, we confirmed that a quadratic Lyapunov function certifies an optimal ellipsoidal region of attraction. Therefore, this example offers limited insight for ROA comparison or showcasing the capabilities of neural Lyapunov functions. We replicated the training of a local neural Lyapunov function as described in \cite{gaby2022lyapunov,grune2021computing}, employing a differential inequality loss and a simple one-layer network. The training concluded within two epochs, achieving a maximum loss below $10^{-6}$ and an average loss under $10^{-7}$, as depicted in Figure \ref{fig:10d}. Due to the system's high dimensionality, we restricted the neural network's complexity to enable efficient verification. The verified ROA is smaller than that achieved with the quadratic Lyapunov function, which is optimal for this case. Training a local Lyapunov function seems straightforward, but capturing the full domain of attraction presents challenges, especially in high-dimensional systems. In the next example, we will demonstrate how these techniques can be applied to such systems.

\begin{figure}[h]
  \centering
  \includegraphics[width=\linewidth]{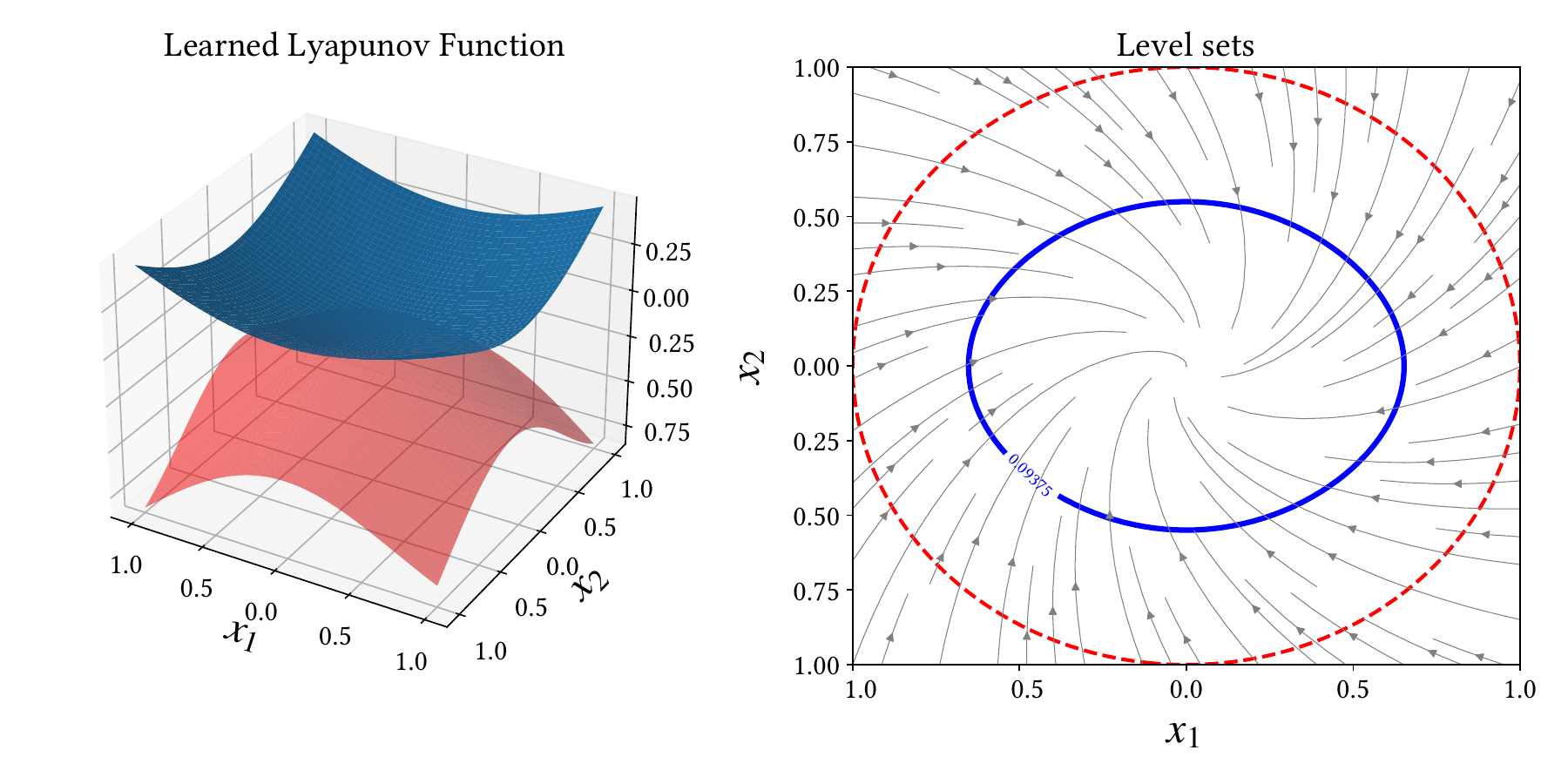}
  \caption{Verified neural Lyapunov function for a 10-dimensional system (Example \ref{ex:10d}). Dashed red: quadratic Lyapunov function; solid blue: neural Lyapunov trained using differential inequality loss as in \cite{grune2021computing,gaby2022lyapunov}. 
  The red surface depicts the derivative of the learned Lyapunov function along the vector field the system. Plots are projected to the $(x_1,x_2)$-plane.} 
  \label{fig:10d}
\end{figure}

\begin{example}[Networked Van der Pol oscillators]
Inspired by \cite{kundu2015sum}, we consider a network of reversed Van der Pol equations of the form
\begin{align*}
\dot x_{i1} &= -x_{i2}, \\
\dot x_{i2} &= x_{i1} - \mu_i (1 - x_{i1}^2)x_{i2} + \sum_{j\neq i} \mu_{ij} x_{i1}x_{j2},
\end{align*}
where $\mu_i$ is a parameter ranging from $(0.5, 2.5)$ and $\mu_{ij}$ represents the interconnection strength. We choose the number of total subsystems $l=10$. The network topology is depicted in Fig. \ref{fig:vdp_network_topology}. The parameters $\mu_i$ are randomly generated and take the following values for $i=1,\ldots,10$:
$$
 [1.25, 2.4, 1.96, 1.7, 0.81, 0.81, 0.62, 2.23, 1.7, 1.92]. 
$$
We set that $\mu_{ij}\in (-0.1,0.1)$ and the number of nonzero entries in $\set{\mu_{ij}}$ for each $i$ is fewer than $3$. 
\end{example}

The total dimension of the interconnected system is therefore 20, which is beyond the capability of current SMT or SDP-based synthesis of Lyapunov functions if a monolithic approach is taken. We trained neural networks of three hidden layers with 10 neurons each for the individual subsystems using the approach proposed in Section \ref{sec:algo}. We then compositionally verified regions of attraction using the approach detailed in Section \ref{sec:verify} (see also \cite{liu2024compositionally}), leveraging the compositional structure of the networked system. The verified regions of attraction for subsystems 1 and 2 are shown in Fig. \ref{fig:vdp_network_level_sets}. It is clear that the neural network Lyapunov functions outperform SOS Lyapunov functions. 

\begin{figure}[ht]
    \centering
    \begin{tikzpicture}[scale=0.9, transform shape]
        \foreach \i in {1,...,10} {
            \node[circle, fill=blue!50] (N\i) at (360/10*\i:2) {\i};
        }
        
\draw[->] (N9) -- (N1);
\draw[->] (N6) -- (N2);
\draw[->] (N9) -- (N2);
\draw[->] (N2) -- (N3);
\draw[->] (N10) -- (N3);
\draw[->] (N3) -- (N4);
\draw[->] (N9) -- (N4);
\draw[->] (N7) -- (N5);
\draw[->] (N10) -- (N5);
\draw[->] (N1) -- (N6);
\draw[->] (N8) -- (N6);
\draw[->] (N3) -- (N7);
\draw[->] (N5) -- (N7);
\draw[->] (N3) -- (N8);
\draw[->] (N7) -- (N8);
\draw[->] (N2) -- (N9);
\draw[->] (N6) -- (N9);
\draw[->] (N4) -- (N10);
\draw[->] (N8) -- (N10);
    \end{tikzpicture}
    \caption{The network topology for the Van der Pol network.}
    \label{fig:vdp_network_topology}
\end{figure}
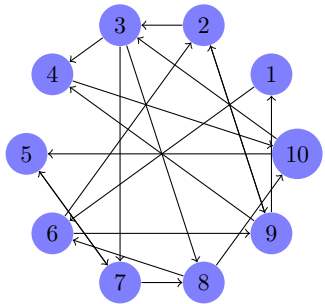

\begin{figure}[!htbp]
    \centering
    \includegraphics[width=0.47\textwidth]{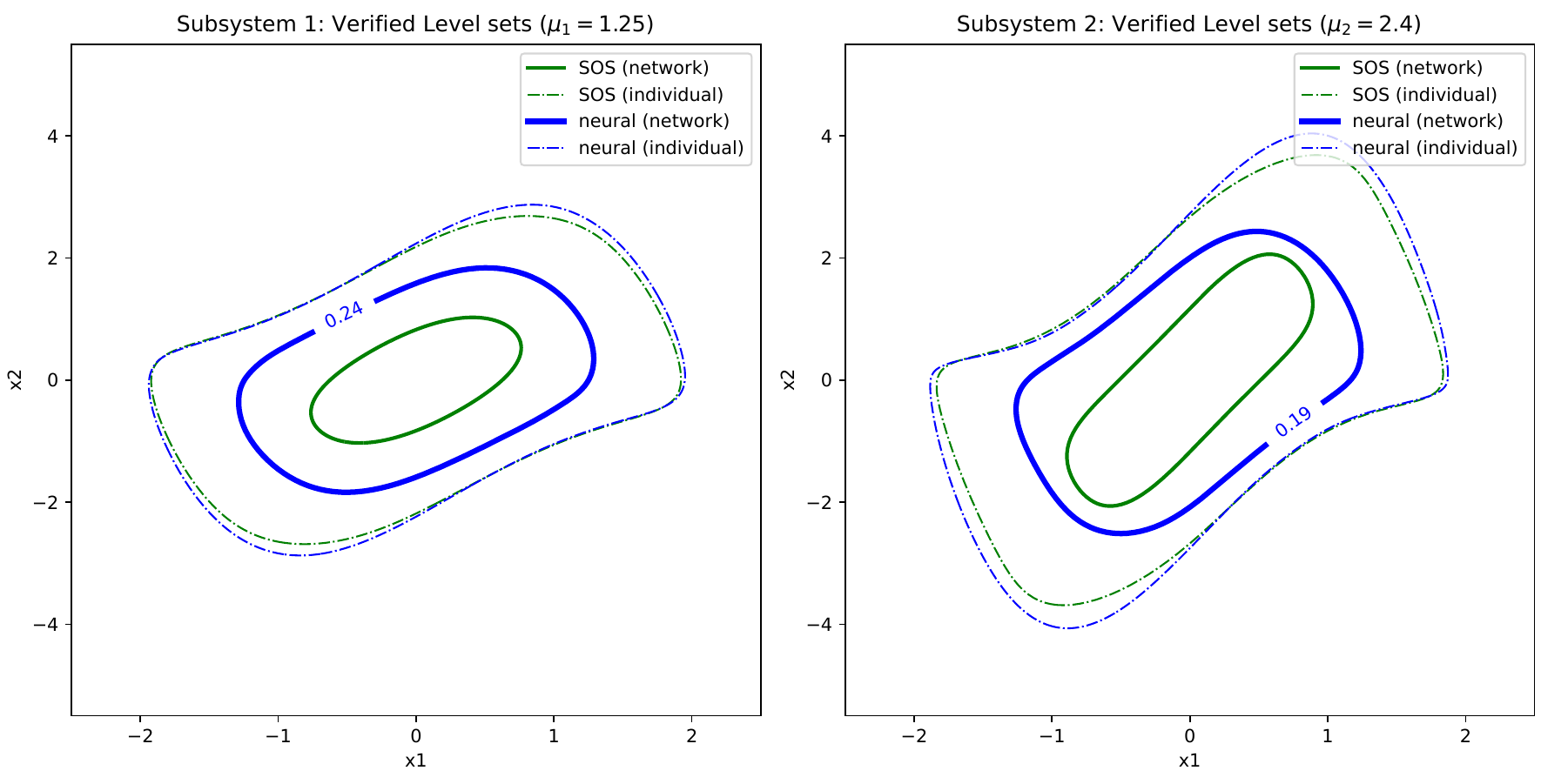}
    \caption{Verified regions of attraction for networked Van der Pol oscillators. We show the results for subsystems 1 and 2 of a 20-dimensional interconnected system. The thick solid lines represent the regions of attraction for the interconnected system, while the thin dashed lines indicate those for individual subsystems. }    \label{fig:vdp_network_level_sets}
\end{figure}

\begin{remark}
Our comparison focuses only on SOS Lyapunov functions because existing approaches using neural Lyapunov functions do not provide approximations of the entire domain of attraction. This limitation exists as training, using the Lyapunov equation (or inequality), must be conducted on a chosen subset of the domain of attraction \cite{chang2019neural,gaby2022lyapunov,grune2021computing} and is inherently local. To the best knowledge of the authors, our work is the first to offer verified regions of attraction that closely approximate the true domain of attraction using neural Lyapunov functions. This is achieved by encoding the Zubov equation (\ref{eq:zubov}) with a PINN approach and allowing training to take place in a region containing the domain of attraction.
\end{remark}

\section{Conclusions} \label{sec:conclude}

We presented a framework for learning neural Lyapunov functions using physics-informed neural networks and verifying them with satisfiability modulo theories solvers. By solving Zubov's PDE with neural networks, we demonstrated that the verified ROA can approach the boundary of the domain of attraction, surpassing sum-of-squares Lyapunov functions obtained through semidefinite programming, as shown by numerical examples. This is achieved by embracing non-convex optimization and leveraging machine learning infrastructure, as well as formal verification tools, to excel where convex optimization may not yield satisfactory results. 

There are numerous ways the framework can be expanded. In the numerical example, we have already demonstrated the potential of using compositional verification to cope with learning and verifying Lyapunov functions for high-dimensional systems. This is ongoing research, and some preliminary results have been reported in \cite{liu2024compositionally}. Future work could also include tool development \cite{liu2024lyznet} with support for different verification engines, such as dReal \cite{gao2013dreal} and Z3 \cite{de2008z3}, and leveraging the growing literature on neural network verification tools. Another natural next step is to handle systems with inputs, including perturbations and controls. Initial work on using physics-informed neural network Lyapunov functions for optimal control has been reported in \cite{meng2024physics}. Investigating robust and control neural Lyapunov functions from the Zubov equation can be an interesting approach \cite{camilli2001generalization}. It would also be interesting to investigate Zubov equation with state constraints to training neural Lyapunov-barrier functions \cite{meng2022smooth} to cope with stability with safety requirements. Simultaneous training of controllers and Lyapunov/value functions are also promising directions, as well as data-driven approaches \cite{meng2023learning}. 

While we have demonstrated that the approach can scale to high-dimensional systems by leveraging the compositional structure, SMT verification of neural networks remains a bottleneck. Further research on efficient tools for formally verifying neural networks can certainly help maximize the full potential of the proposed approach.

\begin{ack}                               
This work was partially supported by the NSERC of Canada and the CRC program. The authors thank the anonymous reviewers for their constructive feedback that helped improve the presentation of the paper. 
\end{ack}

\bibliographystyle{plain}        
\bibliography{automatica} 

\appendix

\section{Proof of Proposition \ref{prop:V}}\label{sec:proof_prop_V}

\begin{proof}
\ifarXiversion
(1) Suppose $x\in \D$. Let $\rho>0$ be from Assumption \ref{as:omega}(ii). There exists some $T_x>0$ such that $\abs{\phi(t,x)}<\rho$ for all $t\ge T_x$. It follows that
\begin{align*}
V(x) & = \int_0^{T_x} \omega(\phi(t,x))dt + \int_{T_x}^\infty  \omega(\phi(t,x))dt    \\
& = \int_0^{T_x} \omega(\phi(t,x))dt + \int_{0}^\infty \omega(\phi(t,\phi(T_x,x)))dt\\
& = \int_0^{T_x} \omega(\phi(t,x))dt + V(\phi(T_x,x)) < \infty,
\end{align*}
where we used finiteness of the first integral and Assumption \ref{as:omega} to conclude $V(\phi(T_x,x))<\infty$. Now suppose that $x\not \in \D$. Then $\phi(t,x)\not\in \D$ for any $t\ge 0$. Since $\D$ is open, there exists some $\delta>0$ such that  $\abs{\phi(t,x)}>\delta$ for all $t\ge 0$. By Assumption \ref{as:omega}, $V(x)=\infty$. 

(2) Let $\set{x_n}$ be a sequence such that  $x_n\ra y$ for some $y\in \partial \D$. Choose $\delta>0$ such that $\bar{\mathcal{B}}_\delta:=\set{x\in\Real^n:\,\abs{x}\le \delta}\subset \D$. Let $T_n$ be the first time that $\abs{\phi(t,x_n)}\le \delta$. Clearly, if $x_n\not\in \D$, then $T_n=\infty$. For all $t\in [0,T_n)$, we  have $\abs{\phi(t,x_n)} > \delta$ and $\omega(\phi(t,x_n))>c$ for some $c>0$ because of Assumption \ref{as:omega}(i). Hence 
$$
V(x_n) = \int_0^\infty \omega(\phi(t,x_n)) dt \ge \int_0^{T_n} \omega(\phi(t,x_n)) dt \ge cT_n.
$$
We can conclude $V(x_n)\ra \infty$ if $T_n\ra\infty$. Suppose that this is not the case. Then $\set{T_n}$ contains a bounded subsequence, still denoted by $\set{T_n}$, that converges to $T$. It follows by continuity that $\phi(T_n,x_n)\ra \phi(T,y)$ and $\abs{\phi(T,y)}\leq\delta$. Hence $\phi(T,y)\in \D$. It follows that $y\in \D$. Since $\D$ is open, this is a contradiction. We must have $V(x_n)\ra\infty$ as $n\ra\ \infty$. 

(3) Positive definiteness of $V$ follows from positive definiteness of $\omega$ and by a continuity argument. 

(4) For $x\in\D$, we have 
\begin{align*}
\dot{V}(x) & = \lim_{t\ra 0+} \frac{V(\phi(t,x))-V(x)}{t}\\
& =\lim_{t\ra 0+}\frac{\int_0^\infty \omega(\phi(s,\phi(t,x)))ds - \int_0^\infty \omega(\phi(s,x))ds}{t} \\
& =  \lim_{t\ra 0+}\frac{\int_0^\infty \omega(\phi(s+t,x))ds -\int_0^\infty \omega(\phi(s,x))ds}{t}\\
& = \lim_{t\ra 0+}\frac{\int_{t}^\infty \omega(\phi(s,x))ds - \int_0^\infty \omega(\phi(s,x))ds}{t}\\
&= \lim_{t\ra 0+} \frac{\int_{t}^0 \omega(\phi(s,x))ds }{t} = -\omega(x).
\end{align*}
\fi
To show continuity of $V$, fix any $x_0\in \D$. By Assumption \ref{as:omega}, for any $\eps>0$, there exists $\delta>0$ such that $\abs{x}<\delta$ implies $V(x)<\eps/2$. By asymptotic stability of the origin, there exists some $T>0$ such that $\abs{\phi(T,x_0)}<\frac{\delta}{2}$. By continuous dependence of solutions to (\ref{eq:sys}) on initial conditions, there exists some $\rho>0$ such that $\B_{\rho}(x_0) \subset \D$ and $\abs{\phi(T,y)}<\delta$ for all $y\in \B_{\rho}(x_0)$. It follows that, for any $y\in \B_{\rho}(x_0)$, 
\begin{align*}
    & \abs{V(y) - V(x_0)} \\
    & = \left\vert\int_0^T \omega(\phi(t,y))dt  - \int_0^T \omega(\phi(t,x_0))dt \right.\\
    & \qquad \left. + \int_T^\infty \omega(\phi(t,y))dt -  \int_T^\infty \omega(\phi(t,x_0))dt \right\vert \\
    & \le \int_0^T \abs{\omega(\phi(t,y))- \omega(\phi(t,x_0))}dt \\
    & \qquad +\abs{V(\phi(T,y)) - V(\phi(T,x_0))}\\
    & \le \int_0^T \abs{\omega(\phi(t,y))- \omega(\phi(t,x_0))}dt + \frac{\eps}{2},
\end{align*}
where the last inequality follows from our choice of $\delta$, $T$, $\rho$, and non-negativeness of $V$. By continuous dependence on initial conditions and continuity of $\omega$, there exists $\rho'\in (0,\rho)$ such that $\int_0^T \abs{\omega(\phi(t,y))- \omega(\phi(t,x_0))}dt<\eps/2$ for all $y\in \B_{\rho'}(x_0)$, which implies 
$$
\abs{V(y) - V(x_0)} < \eps,\quad \forall y\in \B_{\rho'}(x_0). 
$$
We proved that $V$ is continuous on $\D$. 
\end{proof}

\section{Proof of Proposition \ref{prop:lyap}}\label{sec:proof_prop_lyap}

We start with the formal definition of viscosity solutions for first-order PDEs \cite{crandall1983viscosity,bardi1997optimal}. Consider a first-order PDE of the form
\begin{equation}
    \label{eq:pde}
    F(x,u(x),Du(x)) = 0,\quad x\in \Omega,
\end{equation}
where $\Omega\subset\Real^n$ is an open set.

\begin{definition}
    A function $u\in C(\Omega)$ is said to be a viscosity subsolution of (\ref{eq:pde}) if for any  $\psi\in C^1(\Omega)$, we have 
    $$
    F(x,u(x),D\psi(x)) \le 0
    $$
    whenever $x$ is a local maximum of $u-\psi$. It is said to be viscosity supersolution of (\ref{eq:pde}) if for any  $\psi\in C^1(\Omega)$, we have 
    $$
    F(x,u(x),D\psi(x)) \ge 0
    $$
    whenever $x$ is a local minimum of $u-\psi$. We say $u$ is a viscosity solution of (\ref{eq:pde}) if it is both a viscosity subsolution and a viscosity supersolution. 
\end{definition}

An equivalent way to define viscosity solutions is given below. Consider any function $v:\,\Omega\ra \Real$ and define the sets
\begin{align*}
& D^+ v(x) \\
& := \set{p\in\Real^n:\, \limsup_{y\ra x} \frac{v(y)-v(x)-p\cdot(y-x)}{\abs{y-x}}\le 0}
\end{align*}
and 
\begin{align*}
& D^- v(x) \\
& := \set{p\in\Real^n:\, \liminf_{y\ra x} \frac{v(y)-v(x)-p\cdot(y-x)}{\abs{y-x}}\ge 0}. 
\end{align*}
They are called the (Frech\'et) superdifferential and the subdifferential of $v$ at $x$. 

\begin{definition}
    A function $u\in C(\Omega)$ is said to be a viscosity subsolution of (\ref{eq:pde}) if 
    $$
    F(x,u(x),p) \le 0,\quad \forall x\in\Omega, \, p\in D^+u(x), 
    $$
    and a viscosity supersolution of (\ref{eq:pde}) if 
    $$
    F(x,u(x),p) \ge 0,\quad \forall x\in\Omega, \, p\in D^-u(x).
    $$
\end{definition}
The equivalence of these conditions is standard and can be found in \cite[Chapter 2]{bardi1997optimal}. Depending on the situation, one of these equivalent definitions may be more convenient to use. At any point $x$ where $u$ is differentiable, we have $D^+u(x)=D^-u(x)=\set{Du(x)}$ and the PDE (\ref{eq:pde}) is satisfied in the classical sense. 

It is clear from the definition of viscosity solutions that the PDE (\ref{eq:pde}) is not equivalent to $-F(x,u,Du)=0$ in the viscosity sense. When interpreting the Lyapunov PDE (\ref{eq:lyap}) in viscosity sense, we define 
\begin{equation}\label{eq:viscosity_pde}
    F_L(x,u,p) = -\omega(x) - p\cdot f(x).
\end{equation}

We now provide a proof of Proposition \ref{prop:lyap}. 

\begin{proof}
    (1) Consider $V$ defined by (\ref{eq:V}). From Proposition \ref{sec:proof_prop_V}, $V$ is continuous on $\Omega$. Let $\psi\in C^1(\Omega)$. Suppose $x$ is a local maximum point of $V-\psi$. It follows that 
    $$
    V(x) - \psi(x) \ge V(z) - \psi(z)
    $$
    for all $z$ in a small neighborhood of $x$, which implies 
    $$
    V(x) - \psi(x) \ge V(\phi(t,x)) - \psi(\phi(t,x)),
    $$
    for all $t$ close to 0. Rearranging this gives
    $$
    V(\phi(t,x)) - V(x) \le \psi(\phi(t,x)) - \psi(x).
    $$
    By Proposition \ref{prop:V} and continuous differentiability of $\psi$, we have
    \begin{align*}
            -\omega(x) & = \lim_{t\ra 0^+} \frac{V(\phi(t,x)) - V(x)}{t}  \le D\psi(x) \cdot f(x),
    \end{align*}
    which verifies $V$ is a viscosity subsolution of (\ref{eq:lyap}) in view of (\ref{eq:viscosity_pde}). Similarly, we can show that $V$ is a viscosity supersolution. Uniqueness follows from a special case of the optimality principle (see Proposition II.5.18, Theorem III.2.33, and Remark III.2.34 in \cite{bardi1997optimal}). 

    (2) To show local Lipschitz continuity of $V$ under the stated assumptions, consider any compact set $K\subset \D$. Then, for any $r>0$, there exists some $T>0$ such that $\abs{\phi(t,x)}<r$ for all $t\ge T$ and all $x\in K$. There also exists another compact set $\hat K$ such that $\phi(t,x)\in \hat{K}$ for all $t\ge 0$. Indeed, $\hat{K}$ can be taken as the union of the reachable set of (\ref{eq:sys}) from $K$ on $[0,T]$ and $\bar{B}_r$. Let $L_f>0$ be a Lipschitz constant of $f$ on $\hat{K}$ and $L_\omega$ be a Lipschitz constant of $\omega$ on $\hat{K}$. For any $x,y\in K$, we have 
    \begin{align}
        &\abs{V(x)-V(y)} \le  \int_0^T \abs{\omega(\phi(t,x))-\omega(\phi(t,y))}dt \notag\\
        & \quad + \int_T^\infty \abs{\omega(\phi(t,x))-\omega(\phi(t,y))}dt \notag\\
        & \le L_\omega \frac{e^{L_fT}-1}{L_f}\abs{x-y} \notag\\
        & \quad + L_\omega\int_0^\infty \abs{\phi(s,\phi(T,x))-\phi(s,\phi(T,y))}dt, \label{eq:lip}
    \end{align}
    where, to obtain the inequality above, we used Gronwall's inequality to estimate $\abs{\phi(t,x)-\phi(t,y)}\le \abs{x-y}e^{L_f t}$ for $t\in [0,T]$ within the first integral.
    
    We now determine a sufficiently small $r>0$ such that the following contraction condition holds for all solutions of (\ref{eq:sys}) starting from $\B_r$: there exists $C>0$ and $\sigma>0$ such that 
    \begin{equation}
        \label{eq:contraction}
        \abs{\phi(t,x)-\phi(t,y)} \le C e^{-\sigma t} \abs{x-y},\quad \forall x,y\in \B_r. 
    \end{equation}
    We do this by local Lyapunov analysis. By stability of the origin, for any $\eps>0$, there exists some $r>0$ such that $\abs{x}<r$ implies $\abs{\phi(t,x)}>\eps$ for all $t\ge 0$. We focus on our analysis in $\B_\eps$. 
    
    Consider $A=Df(0)$ and $g(x)=f(x)-Ax$. Then $g(0)=Dg(0)=0$. By Assumption \ref{as:f}, $Dg$ is continuous. Let $P$ be the solution to the Lyapunov equation $PA+A^TP=-I$. Consider the Lyapunov function $W(x)=x^TPx$. Fix any $x,y\in \B_r$. Then we have $\phi(t,x),\phi(t,y)\in \B_\eps$ for all $t\ge 0$. Define $E(t)=\phi(t,x)-\phi(t,y)$ for $t\ge 0$. We have 
    \begin{align*}
    \frac{dW(E(t))}{dt} & = 2E^T(t) P (f(t,\phi(t,x) - f(t,\phi(t,y)) \\
    & = 2E^T(t) P (AE(t) + g(\phi(t,x)) - g(\phi(t,y))\\
    & = E^T(t) (PA + A^TP) E(t) \\
    & + 2E^T(t)\int_0^1 Dg(\phi(t,y) + s E(t))dt\cdot E(t)\\
    & \le - \norm{E(t)}^2 + 2\sup_{\abs{z}\le \eps}\norm{Dg(z)} \norm{E(t)}^2, 
    \end{align*}
    where the first two equalities are direct computation, the third equality is by the mean value theorem, and in the last inequality, we used the convexity of $\B_\eps$. Since $Dg$ is continuous and $Dg(0)=0$, we can choose $\eps$ sufficiently small such that $2\sup_{\abs{z}\le \eps}\norm{Dg(z)}<1$. It follows that 
        \begin{align*}
    \frac{dW(E(t))}{dt} \le - c W(E(t)), 
    \end{align*}    
    where $c=(1-2\sup_{\abs{z}\le \eps}\norm{Dg(z)})/\lambda_{\max}(P)$ and $\lambda_{\max}(P)$ is the maximum eigenvalue of $P$. From here we readily conclude that (\ref{eq:contraction}) holds. 

    With (\ref{eq:contraction}), we continue from (\ref{eq:lip}) and compute 
    \begin{align}
        &\abs{V(x)-V(y)} \notag\\
        & \le L_\omega \frac{e^{L_fT}-1}{L_f}\abs{x-y}   + L_\omega e^{L_f T}\abs{x-y} C \int_0^\infty e^{-\sigma t}dt \notag\\
        & = L\abs{x-y}, \notag
    \end{align}    
    where $L=L_\omega (e^{L_fT}-1)/L_f + L_\omega e^{L_f T}C/\sigma$. We have proved that $V$ is locally Lipschitz on $\Omega$. By Rademacher's theorem, $V$ is differentiable almost everywhere in $\D$. At points where $V$ is differentiable, a viscosity solution reduces to a classical solution. Hence, (\ref{eq:lyap}) is satisfied almost everywhere in the classical sense.

    (3) Define 
    \begin{equation}
        \label{eq:J}
        J(x) = \int_0^\infty D\omega (\phi(t,x))\Phi(t,x)dt,\quad x\in \D,
    \end{equation}
    where $\Phi(t)$ is the fundamental matrix solution to the initial value problem
    $$
    \dot{\Phi}(t,x) = A(t,x) \Phi(t,x),\quad \Phi(0)=I,
    $$
    with $I$ being the $n$-dimensional identity matrix and $A(t,x) = Df(\phi(t,x))$. 
    We prove the following:
    \begin{enumerate}[(a)]
        \item $J(x)$ is well defined for all $x\in \D$. 
        \item $J$ is continuous with respect to $x$.
        \item The limit 
        \begin{equation}
            \label{eq:J_C1}
            \lim_{h\ra 0} \frac{\abs{V(x+h)-V(x) - J(x)h}}{\abs{h}} = 0 
        \end{equation}
        holds. 
    \end{enumerate}
    Combining these shows that $V$ is continuously differentiable on $D$ and $DV = J$. 

    To prove (a), fix any $x\in \D$. By asymptotic stability of the origin and $\omega$ being $C^1$, it is clear that $D\omega (\phi(t,x))$ is uniformly bounded. We prove (a) by bounding columns of $\Phi(t,x)$ with Lyapunov analysis. Consider the $i$th column given by 
    $$
    \dot{\Phi}_i(t,x) = A(t,x) \Phi_i(t,x),\quad \Phi_i(0)=e_i. 
    $$
    Clearly, $A(t,x)\ra A=Df(0)$ as $t\ra \infty$, because $f\in C^1$. Consider $P$ and $W$ as defined in the proof of part (2). We have
    \begin{align*}
    \frac{dW(\Phi_i(t,x))}{dt} & = \Phi_i^T(t,x)(P A(t,x) + A^T(t,x)P)\Phi_i(t,x) \\ 
    & = \Phi_i^T(t,x)(P A + A^TP)\Phi_i(t,x) \\
    &\quad + \Phi_i^T(t,x)(P \mathcal{E}(t) + \mathcal{E}^T(t)P)\Phi_i(t,x), 
    \end{align*} 
    where $\mathcal{E}(t)=A(t,x)-A$. Pick $\eps>0$ such that $2\norm{P}\eps<1$. Since $\mathcal{E}(t)\ra 0$ as $t\ra\infty$, there exists some $T>0$ such that $\abs{\mathcal{E}(t)}<\eps$ for all $t\ge T$. It follows that 
    \begin{align*}
    \frac{dW(\Phi_i(t,x))}{dt} & \le - cW(\Phi_i(t,x)), \quad \forall t\ge T, 
    \end{align*}
    where $c=(1-2\norm{P}\eps)/\lambda_{\max}(P)$. It follows that $\Phi_i(t)$ exponentially converges to zero as $t\ra \infty$. Hence $J(x)$ is well defined for each $x\in \D$. 

    To prove (b), we rely on continuous dependence of $\Phi(t,x)$ and $\phi(t,x)$ on $x$. Fix $x\in \D$. Consider $\rho>0$ such that $\bar{\B}_\rho(x)\subset \D$. By the analysis above, we know each element of $\Phi(t,y)$ converges exponentially to zero as $t\ra \infty$ for $y\in \bar{B}_{\rho}(x)$. It follows from a continuity and compactness argument that the convergence is uniform for $y\in \bar{\B}_{\rho}(x)$. Furthermore, $D\omega(\phi(t,y))$ is uniformly bounded for $y\in \bar{\B}_{\rho}(x)$. Similar to the proof of continuity of $V$ in Proposition \ref{prop:V}, we can write 
    \begin{align*}
        & \abs{J(y) - J(x)} \\
        & \le \int_0^T \abs{D\omega(\phi(t,y))\Phi(t,y) - D\omega(\phi(t,x))\Phi(t,x)}dt \\
        &  + \int_T^\infty \abs{D\omega(\phi(t,y))\Phi(t,y)}dt  +   \int_T^\infty \abs{D\omega(\phi(t,x))\Phi(t,x)}dt.
    \end{align*}
    Thanks to uniform convergence of $\Phi(t,y)$ to zero and uniform boundedness of $D\omega(\phi(t,y))$ for $y\in \bar{\B}_{\rho}(x)$, for any $\eps>0$, we can choose $T$ sufficiently large, independent of $y\in \bar{\B}_\rho(x)$, such that the last two integrals are bounded $\eps/4$. For this fixed $T$, by continuous dependence on initial conditions again, we can then choose $\rho'\in(0,\rho)$ such that the first integral is bounded by $\eps/2$ for all $y\in \B_{\rho'}(x)$. It follows that $\abs{J(y)-J(x)}<\eps$ for all $y\in \B_{\rho'}(x)$. We have proved that $J$ is continuous. 

    We now prove (c). Fix any $x\in \D$. Pick $\rho>0$ such that $\bar{\B}_\rho(x)\subset \D$. Let $r>0$ be such that (\ref{eq:contraction}) holds. Choose $\hat T>0$ such that $\abs{\phi(t,y)}<r$ for all $t\ge \hat T$ and all $y\in \bar{\B}_\rho(x)$. Consider any $h$ with $\abs{h}<\rho$ and $T>\hat T$. We have
    \begin{align}
        & \abs{V(x+h) - V(x) - J(x)h} \notag\\
        & = \left\vert\int_0^\infty \omega(\phi(t,x+h))dt - \int_0^\infty \omega(\phi(t,x))dt \right. \notag\\
        & \qquad \left. - \int_0^\infty D\omega (\phi(t,x))\Phi(t,x)dt \cdot h \right\vert \notag\\
        & \le \int_0^T |\omega(\phi(t,x+h)) - \omega(\phi(t,x)) \notag\\
        &\qquad \qquad- D\omega (\phi(t,x))\Phi(t,x) h |dt \notag\\
        & \phantom{\leq} + \int_T^\infty \abs{D\omega (\phi(t,x))\Phi(t,x)}dt \abs{h} \notag\\
        & \phantom{\leq} +  \int_T^\infty \abs{\omega(\phi(t,x+h)) - \omega(\phi(t,x)}dt. \label{eq:est_V_J}
    \end{align}    
    We analyze each of the three integrals above. Let $L_f$ and $L_\omega$ be Lipschitz constants of $f$ and $\omega$ on the reachable set from $\bar{\B}_\rho(x)$. 
    First, by Gronwall's inequality on $[0,\hat T]$ and contraction property (\ref{eq:contraction}) on $[\hat T, T]$, we have 
    $$
    \abs{\phi(T,x+h) - \phi(T,x)} \le e^{L_f \hat T}\abs{h} e^{-\sigma(T-\hat T)}. 
    $$
    By (\ref{eq:contraction}) again, this implies 
    \begin{align}
         & \int_T^\infty \abs{\omega(\phi(t,x+h)) - \omega(\phi(t,x)}dt \notag\\
         & \le L_\omega e^{L_f \hat T}\abs{h} e^{-\sigma(T-\hat T)} C \int_0^\infty e^{-\sigma t}dt \notag\\
         & = \frac{L_\omega e^{L_f \hat T} e^{-\sigma(T-\hat T)} C}{\sigma}\abs{h}. \label{eq:est1}
    \end{align}
    For any $\eps>0$, choose $T$ sufficiently large such that 
    \begin{equation}\label{eq:est2}
    L_\omega e^{L_f \hat T} e^{-\sigma(T-\hat T)} C/\sigma <\eps/4
    \end{equation}
    and 
    \begin{equation}\label{eq:est3}
    \int_T^\infty \abs{D\omega (\phi(t,x))\Phi(t,x)}dt<\frac{\eps}{4},
    \end{equation}
    the latter of which is possible as analyzed in the proof of (b). 

    By the mean value theorem, we have
    \begin{align*}
    & \omega(\phi(t,x+h)) - \omega(\phi(t,x)) = D\omega (\phi(t,\xi))\Phi(t,\xi) h,
    \end{align*}
    where $\xi = (1-c)x + c(x+h)$ for some $c\in (0,1)$. Note that $\xi-x=ch$ By uniform continuity of $D\omega (\phi(\cdot,\cdot))\Phi(\cdot,\cdot)$ on $[0,T]\times \bar{\B}_\rho(x)$, there exists $\rho'\in (0,\rho)$ such that 
    \begin{equation}\label{eq:est4}
    \abs{D\omega (\phi(t,\xi))\Phi(t,\xi) - D\omega (\phi(t,x))\Phi(t,x) } < \frac{\eps}{2},
    \end{equation}
    for all $h$ with $\abs{h}<\rho'$. Putting (\ref{eq:est1}), (\ref{eq:est2}), (\ref{eq:est3}), (\ref{eq:est4}), together with (\ref{eq:est_V_J}) implies 
    \begin{equation}
    \abs{V(x+h) - V(x) - J(x)h} \le \eps \abs{h},
    \end{equation}
    for all $h$ such that $\abs{h}<\rho'$. Hence the limit (\ref{eq:J_C1}) holds and (c) is proved. 

    We conclude that $V$ defined by (\ref{eq:V}) is continuously differentiable on $\D$ and the unique solution solution to (\ref{eq:lyap}) in $C^1(\Omega)$ for any open set $\Omega\subset\D$ by item (1) of the proposition.  
\end{proof}

\section{Proof of Theorem \ref{thm:zubov_pde}}\label{sec:proof_prop_zubov}

A viscosity solution of (\ref{eq:zubov}) is interpreted as a viscosity solution to $F_Z(x,u,Du)=0$ with 
\begin{equation}\label{eq:viscosity_zubov_pde}
    F_Z(x,u,p) = -\omega(x)\psi(u)(1-u) - p\cdot f(x),
\end{equation}
where $\psi$ is from (\ref{eq:beta}). 

\begin{proof}
We first verify that $W$ defined by (\ref{eq:W}) is a viscosity solution to (\ref{eq:zubov}) on $\Real^n$. Let $h\in C^1(\Real^n)$. Suppose that $x_0\in \D$ is a local maximum of $W-h$. It follows from the same argument in the proof of Proposition \ref{prop:lyap}(1) that 
\begin{equation}\label{eq:verify_W1}
    W(\phi(t,x_0)) - W(x_0) \le h(\phi(t,x_0)) - h(x_0)
\end{equation}
for all $t$ close to 0. Note that $\phi(t,x_0)\in \D$ and $W(x)=\beta(V(x))$ for all $x\in\D$. By Proposition \ref{prop:V}, equation (\ref{eq:beta}), and continuous differentiability of $h$, we have
\begin{align}
        -\omega(x_0)\psi(W)(1-W) & = \lim_{t\ra 0^+} \frac{W(\phi(t,x_0)) - W(x_0)}{t}  \notag\\
        & \le Dh(x_0) \cdot f(x_0). \label{eq:verify_W2}
\end{align}
When $x_0\in\Real^n\setminus \D$ is a local maximum of $W-h$, we have $W(\phi(t,x_0))\equiv 1$ and (\ref{eq:verify_W2}) still holds. This verifies that $W$ is a viscosity subsolution of (\ref{eq:lyap}) on $\Real^n$ in view of (\ref{eq:viscosity_zubov_pde}). Similarly, we can show that $W$ is a viscosity supersolution on $\Real^n$. 

We proceed to prove uniqueness on any open set $\Omega$ containing the origin. Let $u_1$ and $u_2$ be two viscosity solutions to (\ref{eq:zubov}) on $\Omega$ subject to the same boundary condition (\ref{eq:zubov_boundary}) and satisfying $u_1(0)=u_2(0)=0$. Pick $\delta>0$ sufficiently small such that $\bar\B_\delta \subset\Omega \cap \D$, $u_i(x)<1$, and $\psi(u_i(x))>0$ for all $x\in \B_\delta$. 
Then $\beta^{-1}$ is well defined and increasing on the range of $u_i(x)$ on $\B_\delta$ for $i=1,2$. This is possible by continuity of $u_i$ and $\psi$, and the definition of $\beta$ in (\ref{eq:beta}). By \cite[Proposition 2.5] {bardi1997optimal}, $v_i(x)=\beta^{-1}(u_i(x))$ is a viscosity solution of 
$$
F_Z(x,\beta(v),\beta'(v)Dv) = 0,\quad x\in \B_\delta. 
$$
By (\ref{eq:viscosity_zubov_pde}) and positiveness of $\psi(\beta(v))(1-\beta(v))$ for $x\in \B_\delta$, $v_i$ ($i=1,2$) is a viscosity solution of (\ref{eq:lyap}) on $\B_\delta$ in view of (\ref{eq:viscosity_pde}). Clearly, $v_1(0)=v_2(0)$. By Proposition \ref{prop:lyap}(1) and continuity, $v_1$ and $v_2$ are identical on $\bar\B_\delta$, which implies $u_1$ and $u_2$ are identical on $\bar\B_\delta$. 

Suppose $u_1$ and $u_2$ are nonidentical on $\Omega'=\Omega\setminus\bar\B_\delta$. Note that $\partial \Omega' = \partial \Omega \cup \partial \bar\B_\delta$. We have $u_1=u_2$ on $\partial \Omega'$. 

Without loss of generality, assume there exists $\bar x\in \Omega'$ such that $u_1(\bar x) > u_2(\bar x)$. For each $\eps>0$, define an  auxiliary function
$$
\Phi_\eps(x,y) = u_1(x) - u_2(y) - \frac{\abs{x-y}^2}{2\eps},\quad (x,y)\in \bar\Omega' \times \bar\Omega'. 
$$
Let $(x_\eps,y_\eps)$ be a maximum $\Phi_\eps$ on $\bar\Omega' \times \bar\Omega'$.  
\ifarXiversion
Since $\Phi_\eps(x_\eps,x_\eps)\le \Phi_\eps(x_\eps,y_\eps)$, we have 
\begin{equation}\label{eq:u2}
\frac{\abs{x_\eps - y_\eps}^2}{2\eps} \le u_2(x_\eps) - u_2(y_\eps).     
\end{equation}
It follows that $\abs{x_\eps - y_\eps}\le \sqrt{C\eps}$, where $C>0$ satisfies $4\sup_{x\in \bar\Omega'}u_2(x)\le C$. Hence $\abs{x_\eps - y_\eps}\ra 0$ as $\eps\ra 0$. By uniform continuity of $u_2$ and (\ref{eq:u2}), $\frac{\abs{x_\eps - y_\eps}^2}{2\eps}\ra 0$ as $\eps\ra 0$. 

We consider two cases: (1) there exists a sequence  $\eps_n\ra 0$ such that either $x_{\eps_n}\in \partial \Omega'$ or $y_{\eps_n} \in \partial \Omega'$; (2) for all $\eps$ sufficiently small, $(x_\eps,x_\eps)\in\Omega'\times\Omega'$. In case (1), if $x_{\eps_n}\in \partial \Omega'$, by the boundary condition, we have 
$$
u_1(x_{\eps_n}) - u_2(y_{\eps_n}) = u_2(x_{\eps_n}) - u_2(y_{\eps_n}), 
$$
which implies
\begin{align*}
0< u_1(\bar x) - u_2(\bar x)  = \Phi_{\eps_n}(\bar x,\bar x) & \le \Phi_\eps(x_{\eps_n},y_{\eps_n})\\
&\le u_1(x_{\eps_n}) - u_2(y_{\eps_n}) \\
& = u_2(x_{\eps_n}) - u_2(y_{\eps_n}). 
\end{align*}
Letting $n\ra \infty$, the right-hand side approaches zero, which is a contradiction. For the case with $y_{\eps_n}\in \partial \Omega'$, we have 
$$
u_1(x_{\eps_n}) - u_2(y_{\eps_n}) = u_1(x_{\eps_n}) - u_1(y_{\eps_n}), 
$$
which leads to a similar contradiction. 

Now we consider the case $(x_\eps,y_\eps)\in\Omega'\times\Omega'$ for all $\eps$ sufficiently small. Define 
$$
h_1(y) = u_1(x_\eps) - \frac{\abs{x_\eps - y}^2}{2\eps},\quad h_2(x) = u_2(y_\eps) + \frac{\abs{x - y_\eps}^2}{2\eps}. 
$$
It follows that $x_\eps$ is a local maximum of $u_1-h_2$ and $y_\eps$ is a local minimum of $u_2-h_1$. Clearly, $h_1$ and $h_2$ are continuously differentiable and 
$$
Dh_1(y_\eps) = \frac{x_\eps - y_\eps}{\eps} = Dh_2(x_\eps). 
$$
By the definition of viscosity sub- and supersolutions for (\ref{eq:viscosity_zubov_pde}), we have
\fi
$$
-\omega(x_\eps)\psi(u_1(x_\eps))(1-u_1(x_\eps)) - \frac{x_\eps - y_\eps}{\eps}\cdot f(x_\eps) \le 0,
$$
and 
$$
-\omega(y_\eps)\psi(u_2(y_\eps))(1-u_2(y_\eps)) - \frac{x_\eps - y_\eps}{\eps}\cdot f(y_\eps) \ge 0.
$$
Let 
$$
a_\eps =  - \frac{(x_\eps - y_\eps)\cdot f(x_\eps)}{\eps\omega(x_\eps)}, \quad b_\eps = - \frac{(x_\eps - y_\eps)\cdot f(y_\eps)}{\eps\omega(y_\eps)}.
$$
Define $G(s)=\psi(s)(1-s)$. From the above two inequalities, we have
$$
G(u_1(x_\eps))\ge a_\eps,\quad G(u_2(y_\eps))\le b_\eps.  
$$
Furthermore, we have 
\begin{align}
0&<u_1(\bar x) - u_2(\bar x) = \Phi_\eps(\bar x,\bar x) \le \Phi_\eps(x_\eps,y_\eps)\notag\\
& \le u_1(x_\eps)-u_2(y_\eps). \label{eq:u1-u2}
\end{align}
By monotonicity of $G$, we obtain  
$$
a_\eps - b_\eps \le G(u_1(x_\eps)) -  G(u_2(y_\eps)) \le 0. 
$$
Recall that $\abs{x_\eps-y_\eps}\ra 0$ and $\frac{\abs{x_\eps - y_\eps}^2}{2\eps}\ra 0$ as $\eps\ra 0$. By Lipschitz continuity of $f$ and $\omega$ on $\bar\Omega'$, we can verify $a_\eps-b_\eps\ra 0$ as $\eps \ra 0$. Now, by uniform continuity of $G^{-1}$ on the compact set $G(Y)$, where $Y\subset I$ is a compact set containing the image of $u_1$ and $u_2$ from $\bar\Omega'$, we conclude that $u_1(x_\eps)-u_2(y_\eps)\ra 0$ as $\eps\ra 0$, which contradicts (\ref{eq:u1-u2}). 

Hence $u_1$ and $u_2$ are identical on $\Omega'$ as well. We have proved uniqueness of viscosity solution to (\ref{eq:zubov}) on $\bar\Omega$. Items (2) and (3) are direct consequences of items (2) and (3) in Proposition \ref{prop:lyap}. 
\end{proof}

\section{Proof of Proposition  \ref{prop:error}}\label{sec:proof_prop_error}

We begin with the following definition. We say that $v$ is an $\eps$-approximate viscosity solution of (\ref{eq:zubov}), if 
    \begin{align*}
    F_Z(x,v(x),p) &\le \eps,\quad \forall x\in\Omega, \, p\in D^+v(x), \\
    F_Z(x,v(x),p) &\ge -\eps,\quad \forall x\in\Omega, \, p\in D^-v(x). 
    \end{align*}
    where $F_Z$ is from (\ref{eq:viscosity_zubov_pde}).

\begin{proof}
    For each $\delta>0$ such that $\bar\B_\delta \subset\Omega \cap \D$, we denote $\Omega'=\Omega\setminus\bar\B_\delta$. Note that $\partial \Omega' = \partial \Omega \cup \partial \bar\B_\delta$. In the proof of Theorem \ref{thm:zubov_pde}, we essentially proved the following fact: 
    if $u_1$ is a viscosity subsolution to (\ref{eq:zubov}) and $u_2$ is a viscosity supersolution to (\ref{eq:zubov}) on  $\Omega'$ and $u_1\le u_2$ on $\partial \Omega'$, then we have $u_1\le u_2$ on $\Omega'$. We define 
    \begin{align}
    C(\eps, \eps_b) & = \max\left\{ -v + G^{-1}\left(G(v) - \frac{\eps}{\min_{x\in\bar\Omega\setminus\B_{\delta}}\omega(x)}\right),\right. \notag\\
    & \quad \left. v - G^{-1}\left(G(v) + \frac{\eps}{\min_{x\in\bar\Omega\setminus\B_{\delta}}\omega(x)}\right), \eps_b \right\}\label{eq:C}
    \end{align}
    Evaluations of $G^{-1}$ above are valid because of the assumption on $v$ and $\eps$. 
    By continuity of $G$ and $G^{-1}$, $C(\eps,\eps_b)\ra 0$ as $\eps\ra 0$ and $\eps_b\ra 0$. With this choice, it follows from monotonicity of $G$ that 
    \begin{align*}
     -\omega(x) G(v+C(\eps,\eps_b)) &\ge -\omega(x) \left(G(v) - \frac{\eps}{\min_{x\in\bar\Omega\setminus\B_{\delta}}\omega(x)}\right)\\
    &  \ge -w(x)G(v) + \eps.
    \end{align*}
    Similarly, $-\omega(x) G(v-C(\eps,\eps_b)) \le -w(x)G(v) - \eps.$ 
    Hence, we can readily verify that $v+C(\eps,\eps_b)$ is a viscosity supersolution and $v-C(\eps,\eps_b)$ is a viscosity subsolution of (\ref{eq:zubov}) on $\Omega'$ with boundary condition (\ref{eq:zubov_boundary}) on $\partial\Omega'$. By the comparison argument in the proof of Theorem \ref{thm:zubov_pde}, we can show $W(x)\le v+C(\eps,\eps_b)$ and $W(x)\ge v-C(\eps,\eps_b)$ on $\Omega'$. 

    Now consider the sequence $\set{v_n}$ and we show it uniformly converges to $W$ on $\bar\Omega$. For any $\rho>0$, choose $\delta>0$ and $N>0$ such that $\abs{W(x)}\le \frac{\rho}{2}$ and $\abs{v_n(x)}\le \frac{\rho}{2}$ for all $x\in\B_\delta$ and all $n>N$. This is possible by the fact that $W(0)=0$, $W$ is continuous, $v_n(0)\ra 0$ as $n\ra \infty$ and $v_n$ has a uniform Lipschitz constant on $\Omega$. By uniform convergence of $v_n$ to $W$ on $\partial \Omega$, we can choose $N$ sufficiently large such that $\abs{W(x)-v_n(x)}\le \rho$ for all $n>N$ and $x\in\partial \Omega$. 

    With fixed $\rho$ and $\delta$, choose $N$ sufficiently large such that $C(\eps_n,\rho)\le \rho$ for all $n>N$. This is possible by $\eps_n\downarrow 0$, the continuity of $G$ and $G^{-1}$, and the definition of $C(\eps,\eps_b)$ in (\ref{eq:C}). Since $v_n$ is an $\eps_n$-approximate viscosity solution to (\ref{eq:zubov}) on $\Omega'$ with boundary error $\abs{W(x)-v_n(x)}\le \rho$ on $\partial\Omega'$, where $\Omega'=\Omega\setminus\bar\B_\delta$, by the conclusion established earlier, we have 
    $
    \abs{v_n(x)-W(x)}\le C(\eps_n, \rho) \le \rho
    $
    for all $x\in \bar\Omega'$. Since we also have $\abs{W(x)-v_n(x)}\le \rho$ for $x\in \B_\delta$ and $n>N$, we have proved $\abs{v_n(x)-W(x)}\le \rho$ for all $x\in \bar\Omega$ and all $n>N$. 
\end{proof}

\end{document}